\documentclass[12pt,a4paper]{amsart}
\usepackage{amssymb}
\usepackage[all,cmtip]{xy}

\calclayout

\newcommand{\+}{\protect\nobreakdash-}

\renewcommand{\:}{\colon}

\newcommand{\rarrow}{\longrightarrow}

\newcommand{\ot}{\otimes}

\newcommand{\lrarrow}{\mskip.5\thinmuskip\relbar\joinrel\relbar\joinrel
 \rightarrow\mskip.5\thinmuskip\relax}

\newcommand{\bu}{{\text{\smaller\smaller$\scriptstyle\bullet$}}}

\DeclareMathOperator{\Hom}{Hom}
\DeclareMathOperator{\Ext}{Ext}
\DeclareMathOperator{\Tor}{Tor}
\DeclareMathOperator{\Id}{Id}
\DeclareMathOperator{\id}{id}

\newcommand{\FP}{\mathrm{FP}}
\newcommand{\kappaP}{\kappa\text{-}\mathrm{P}}

\newcommand{\Modl}{{\operatorname{\mathsf{--Mod}}}}
\newcommand{\Modr}{{\operatorname{\mathsf{Mod--}}}}
\newcommand{\Modrfl}{{\operatorname{\mathsf{Mod_{fl}--}}}}
\newcommand{\Modrkapfl}{{\operatorname{\mathsf{Mod_{\kappa-fl}--}}}}

\newcommand{\Sets}{\mathsf{Sets}}
\newcommand{\Ab}{\mathsf{Ab}}
\newcommand{\Funct}{\mathsf{Funct}}
\newcommand{\Fil}{\mathsf{Fil}}

\newcommand{\fp}{\mathsf{fp}}
\newcommand{\ad}{\mathsf{ad}}

\newcommand{\sop}{{\mathsf{op}}}

\newcommand{\cM}{\mathcal M}
\newcommand{\cS}{\mathcal S}
\newcommand{\cT}{\mathcal T}

\newcommand{\sA}{\mathsf A}
\newcommand{\sB}{\mathsf B} 
\newcommand{\sC}{\mathsf C}
\newcommand{\sD}{\mathsf D}
\newcommand{\sE}{\mathsf E}
\newcommand{\sF}{\mathsf F}
\newcommand{\sG}{\mathsf G}
\newcommand{\sJ}{\mathsf J}
\newcommand{\sK}{\mathsf K}
\newcommand{\sL}{\mathsf L}
\newcommand{\sM}{\mathsf M}
\newcommand{\sS}{\mathsf S}
\newcommand{\sT}{\mathsf T}
\newcommand{\sX}{\mathsf X}
\newcommand{\sY}{\mathsf Y}

\newcommand{\bC}{\mathbf C}
\newcommand{\bH}{\mathbf H}
\newcommand{\bD}{\mathbf D}

\newcommand{\boZ}{\mathbb Z}

\newcommand{\Section}[1]{\bigskip\section{#1}\medskip}
\theoremstyle{plain}
\newtheorem{thm}{Theorem}[section]
\newtheorem*{thm0}{Theorem 0}
\newtheorem*{thma}{Theorem A}
\newtheorem*{thmb}{Theorem B}
\newtheorem*{propa}{Proposition A}
\newtheorem*{propb}{Proposition B}
\newtheorem{lem}[thm]{Lemma}
\newtheorem{prop}[thm]{Proposition}
\newtheorem{cor}[thm]{Corollary}
\theoremstyle{definition}

\newtheorem{rem}[thm]{Remark}

\begin{document}

\title{Generalized periodicity theorems}

\author{Leonid Positselski}

\address{Institute of Mathematics, Czech Academy of Sciences \\
\v Zitn\'a~25, 115~67 Praha~1, Czech Republic}

\email{positselski@math.cas.cz}

\begin{abstract}
 Let $R$ be a ring and $\sS$ be a class of strongly finitely presented
($\FP_\infty$) $R$\+modules closed under extensions, direct summands,
and syzygies.
 Let $(\sA,\sB)$ be the (hereditary complete) cotorsion pair generated
by $\sS$ in $\Modr R$, and let $(\sC,\sD)$ be the (also hereditary
complete) cotorsion pair in which $\sC=\varinjlim\sA=\varinjlim\sS$.
 We show that any $\sA$\+periodic module in $\sC$ belongs to $\sA$,
and any $\sD$\+periodic module in $\sB$ belongs to~$\sD$.
 Further generalizations of both results are obtained, so that we get
a common generalization of the flat/projective and fp\+projective
periodicity theorems, as well as a common generalization of
the fp\+injective/injective and cotorsion periodicity theorems.
 Both are applicable to modules over an arbitrary ring, and in fact,
to Grothendieck categories.
\end{abstract}

\maketitle

\tableofcontents

\section*{Introduction}
\medskip

\setcounter{subsection}{-1}

\subsection{{}} \label{introd-seven-items}
 \emph{Periodicity theorems} in homological algebra apply to
the following setup.
 Let $R$ be an associative ring and
\begin{equation} \label{periodicity-sequence}
 0\lrarrow M\lrarrow L\lrarrow M\lrarrow0 \tag{$*$}
\end{equation}
be a short exact sequence of (right) $R$\+modules with
the leftmost term isomorphic to the rightmost one.
 Then it is known that
\begin{enumerate}
\item if the $R$\+module $M$ is flat and the $R$\+module $L$ is
projective, then the $R$\+mod\-ule $M$ is projective
(Benson and Goodearl 2000~\cite{BG}, rediscovered by
Neeman in~2008~\cite{Neem});
\item if the exact sequence~\eqref{periodicity-sequence} is pure and
the $R$\+module $L$ is pure-projective, then the $R$\+module $M$ is
pure-projective (Simson 2002~\cite{Sim});
\item if the exact sequence~\eqref{periodicity-sequence} is pure and
the $R$\+module $L$ is pure-injective, then the $R$\+module $M$ is
pure-injective (\v St\!'ov\'\i\v cek 2014~\cite{Sto});
\item in particular, if the $R$\+module $M$ is fp\+injective and
the $R$\+module $L$ is injective, then the $R$\+module $M$ is injective;
\item if the $R$\+module $L$ is cotorsion, then the $R$\+module $M$ is
cotorsion (Bazzoni, Cort\'es-Izurdiaga, and Estrada 2017~\cite{BCE});
\item if the ring $R$ is right coherent and the right $R$\+module $L$
is fp\+projective, then the $R$\+module $M$ is fp\+projective
(\v Saroch and \v St\!'ov\'\i\v cek 2018~\cite{SarSt});
\item over any ring $R$, if the $R$\+module $L$ is fp\+projective,
then the $R$\+module $M$ is weakly fp\+projective
(Bazzoni, Hrbek, and the present author 2022~\cite{BHP}).
\end{enumerate}

 Periodicity phenomena are linked to behavior of the modules of
cocycles in acyclic complexes.
 This means that the assertions~(1\+-7) can be restated as follows:
\begin{enumerate}
\renewcommand{\theenumi}{\arabic{enumi}$^{\text{c}}$}
\item in any acyclic complex of projective modules with flat modules
of cocycles, the modules of cocycles are actually projective (so
the complex is contractible);
\item in any pure acyclic complex of pure-projective modules,
the modules of cocycles are pure-projective (so the complex is
contractible);
\item in any pure acyclic complex of pure-injective modules,
the modules of cocycles are pure-injective (so the complex is
contractible);
\item in any acyclic complex of injective modules with fp\+injective
modules of cocycles, the modules of cocycles are actually injective
(so the complex is contractible);
\item in any acyclic complex of cotorsion modules, the modules of
cocycles are cotorsion;
\item in any acyclic complex of fp\+projective right modules over
a right coherent ring, the modules of cocycles are fp\+projective;
\item in any acyclic complex of fp\+projective modules (over any ring),
the modules of cocycles are weakly fp\+projective.
\end{enumerate}
 We refer to the introduction to the paper~\cite{BHP} for 
a more detailed discussion of the periodicity
theorems~(1\+-7) and~(1$^\text{c}$\+-7$^\text{c}$).

\subsection{{}} \label{introd-theorem-zero}
 The aim of this paper is to obtain a common generalization
of~(1) and~(6\+-7), and also a common generalization of~(4) and~(5),
in the context of a chosen class of modules or objects in
a Grothendieck category.
 Let us start with presenting the most symmetric and nicely looking
formulation of a special case of our main results, and then proceed
to further generalizations.

 Let $R$ be a ring.
 An $R$\+module is said to be \emph{strongly finitely presented}
if it has an (infinite) resolution by finitely generated projective
$R$\+modules.
 In the terminology of the book~\cite{GT}, such modules are called
``$\FP_\infty$\+modules''.

 Let $\sS$ be a class (up to an isomorphism, of course, a set) of
strongly finitely presented (right) $R$\+modules.
 Assume that the free $R$\+module $R$ belongs to $\sS$, and that
the class of modules $\sS$ is closed under direct summands,
extensions, and kernels of epimorphisms in $\Modr R$.
 In particular, for any module $S\in\sS$ there exists a (finitely
generated) projective $R$\+module $P$ together with an $R$\+module
epimorphism $P\rarrow S$ whose kernels also belongs to~$\sS$.
 The latter property is expressed by saying that ``the class
of modules ~$\sS$ is closed under syzygies''.

 Denote by $\sB=\sS^{\perp_1}$ the class of all $R$\+modules $B$
such that $\Ext^1_R(S,B)=0$ for all $S\in\sS$.
 Furthermore, denote by $\sA={}^{\perp_1}\sB$ the class of all
$R$\+modules $A$ such that $\Ext^1_R(A,B)=0$ for all $B\in\sB$.
 The pair of classes of modules $(\sA,\sB)$ is called
\emph{the cotorsion pair generated by\/ $\sS$ in\/ $\Modr R$}.

 Let $\sC=\varinjlim\sS$ denote the class of all $R$\+modules that
can be obtained as direct limits of diagrams of modules from $\sS$,
indexed by directed posets.
 Since $\sS$ is a class of finitely presented modules, one can see
that $\varinjlim\sS$ coincides with the direct limit closure of
$\sS$ in $\Modr R$ \,\cite{Len,CB,Kra}.

 Furthermore, since $\sS$ is a class of strongly finitely presented
modules ($\FP_2$ is sufficient), the class $\sC$ is closed under 
extensions in $\Modr R$ \,\cite{AT}.
 Taking into account the description of $\sA$ as the class of all
direct summands of transfinitely iterated extensions of modules
from $\sS$ \,\cite[Corollary~6.14]{GT}, one concludes that
$\sA\subset\sC$.
 Hence $\sC=\varinjlim\sA$ is the class of all direct limits of
modules from~$\sA$.

 Denote by $\sD=\sC^{\perp_1}$ the class of all $R$\+modules $D$
such that $\Ext^1_R(C,D)=0$ for all $C\in\sC$.
 Then one has $\sA\subset\sC$ and $\sB\supset\sD$.

 Part~(a) of the following theorem is one of the main results of this
paper, while part~(b) follows rather easily from a result of Bazzoni,
Cort\'es-Izurdiaga, and Estrada~\cite[Theorem~4.7]{BCE} together with
a result of Angeleri H\"ugel and Trlifaj~\cite[Corollary~2.4]{AT}.

\begin{thm0}
 Let $R$ be a ring and\/ $\sS$ be a class of (strongly) finitely
presented $R$\+modules, containing the free $R$\+module $R$ and
closed under direct summands, extensions, and kernels of epimorphisms.
 Put\/ $\sB=\sS^{\perp_1}$, \ $\sA={}^{\perp_1}\sB$, \
$\sC=\varinjlim\sS$, and\/ $\sD=\sC^{\perp_1}$.
 Then the following assertions hold: \par
\textup{(a)} For any short exact sequence~\eqref{periodicity-sequence}
with $L\in\sA$ and $M\in\sC$, one has $M\in\sA$.
 In other words, in any acyclic complex of modules from\/ $\sA$ with
the modules of cocycles belonging to\/ $\sC$, the modules of cocycles
actually belong to\/~$\sA$. \par
\textup{(b)} For any short exact sequence~\eqref{periodicity-sequence}
with $L\in\sD$ and $M\in\sB$, one has $M\in\sD$.
 In other words, in any acyclic complex of modules from\/ $\sD$ with
the modules of cocycles belonging to\/ $\sB$, the modules of cocycles
actually belong to\/~$\sD$.
\end{thm0}

 Theorem~0(a) is a common generalization of items~(1)
or~(1$^\text{c}$) and~(6) or~(6$^\text{c}$) on
the list of Section~\ref{introd-seven-items}.
 Taking $\sS$ to be the class of all finitely generated projective
$R$\+modules, one obtains the flat/projective periodicity theorem
of Benson--Goodearl~\cite[Theorem~2.5]{BG} and
Neeman~\cite[Remark~2.15]{Neem} as a particular case of Theorem~0(a).
 Assuming the ring $R$ to be right coherent and taking $\sS$ to be
the class of all finitely presented right $R$\+modules, one obtains
the fp\+projective periodicity theorem of \v Saroch and
\v St\!'ov\'\i\v cek~\cite[Example~4.3]{SarSt} as a particular case
of Theorem~0(a).

 Theorem~0(b) is a common generalization of items~(4)
or (4$^\text{c}$) (for coherent rings) and~(5)
or~(5$^\text{c}$) on the list of Section~\ref{introd-seven-items}.
 Assuming the ring $R$ to be right coherent and taking $\sS$ to be
the class of all finitely presented right $R$\+modules, one obtains
the fp\+injective/injective periodicity theorem, essentially due to
\v St\!'ov\'\i\v cek~\cite[Corollary~5.5]{Sto} (see
also~\cite[Theorem~1.2(1) or~5.1(1)]{BCE}), as a particular case
of Theorem~0(b).
 Taking $\sS$ to be the class of all finitely generated projective
$R$\+modules, one obtains the cotorsion periodicity theorem of Bazzoni,
Cort\'es-Izurdiaga, and Estrada~\cite[Theorem~1.2(2) or~5.1(2)]{BCE}
as a particular case of Theorem~0(b).

\subsection{{}} \label{introd-theorem-A}
 Both parts~(a) and~(b) of Theorem~0 admit far-reaching generalizations
in several directions simultaneously (allowing, in particular, to
drop the coherence assumptions on the ring $R$ in the preceding two
paragraphs).
 Let us state these results.

 We consider a Grothendieck abelian category~$\sK$.
 For any class of objects $\sS\subset\sK$, let
$\varinjlim\sS\subset\sK$ denote the class of all direct limits in $\sK$
of diagrams of objects from $\sS$ indexed by directed posets.
 Furthermore, given a regular cardinal~$\kappa$, we let
$\varinjlim^{(\kappa)}\sS\subset\sK$ denote the class of all direct
limits in $\sK$ of diagrams of objects from $\sS$ indexed by
\emph{$\kappa$\+directed} posets.
 Here a poset $X$ is said to be $\kappa$\+directed if each of its
subsets of cardinality less than~$\kappa$ has an upper bound in~$X$.

 The following theorem is the main result of this paper,
formulated in full generality and strength.
 It is a generalization of Theorem~0(a).

\begin{thma}
 Let\/ $\sK$ be a Grothendieck abelian category, and let $\kappa$~be
a regular cardinal such that\/ $\sK$ is a locally $\kappa$\+presentable
category.
 Let\/ $\sS\subset\sK$ be a class of (some) $\kappa$\+presentable
objects closed under transfinitely iterated extensions indexed by
ordinals smaller than~$\kappa$.
 Put\/ $\sC=\varinjlim^{(\kappa)}\sS\subset\sK$, and assume that
the class\/ $\sC$ is deconstructible in\/~$\sK$.
 Denote by\/ $\sA=\Fil(\sS)^\oplus$ the class of all direct summands of
transfinitely iterated extensions of objects from\/ $\sS$ in\/~$\sK$.
 Let\/ $\sB'=\sS^{\perp_{\ge1}}\cap\sC$ be the class of all objects
$B\in\sC$ such that\/ $\Ext_\sK^n(S,B)=0$ for all $S\in\sS$ and $n\ge1$.
 Let\/ $\sA'=\sC\cap{}^{\perp_1}\sB'$ be the class of all objects
$A\in\sC$ such that\/ $\Ext_\sK^1(A,B)=0$ for all $B\in\sB'$
(so\/ $\sA\subset\sA'\subset\sC$).
 Then, in any acyclic complex of objects from\/ $\sA$ with the objects
of cocycles belonging to\/ $\sC$, the objects of cocycles actually
belong to\/~$\sA'$.
\end{thma}

 The following proposition applies in the case of the countable
cardinal $\kappa=\aleph_0$.
 It is a supplementary comment on Theorem~A, providing some sufficient
conditions for validity of the main assumption of the theorem.

\begin{propa}
 Let\/ $\sK$ be a locally finitely presentable abelian category,
and let\/ $\sS\subset\sK$ be a class of (some) finitely presentable
objects closed under extensions in\/~$\sK$.
 Put\/ $\sC=\varinjlim\sS\subset\sK$.
 Then \par
\textup{(i)} If the class\/ $\sS$ consists of (some) objects of type\/
$\FP_2$, then the class\/ $\sC$ is closed under extensions in\/~$\sK$.
\par
\textup{(ii)} If the class\/ $\sC$ is closed under extensions in\/ $\sK$,
then the class\/ $\sC$ is deconstructible in\/~$\sK$.
\end{propa}

 Let us empasize that deconstructibility in Theorem~A and Proposition~A
is understood in the strong sense of the word.
 So, a class of objects $\sC$ is \emph{deconstructible} if it is closed
under transfinitely iterated extensions and contains a subset $\sT
\subset\sC$ such that all objects from $\sC$ are transfinitely iterated
extensions of objects from~$\sT$.

 Taking $\kappa=\aleph_0$ and assuming $\sS$ to be a class of strongly
finitely presentable ($\FP_\infty$) objects closed under extensions
in $\sK$ makes the assertions of Proposition~A applicable, so
the assumption of Theorem~A is satisfied.
 This makes Theorem~0(a) a~particular case of Theorem~A with
Proposition~A (for $\sK=\Modr R$).

 Taking $\sS$ to be the class of all $\kappa$\+presentable objects
in $\sK$ makes the assumption of Theorem~A satisfied as well,
since $\sC=\sK$ in this case.
 So one obtains the theorem of Bazzoni, Hrbek, and the present
author (\cite[Theorem~4.1]{BHP} for $\kappa=\aleph_0$, listed above
as item~(7) or~(7$^\text{c}$) in the case of $\sK=\Modr R$;
or~\cite[Remark~4.11]{BHP} for other~$\kappa$) as a particular case
of Theorem~A.

 For another particular case of Theorem~A with Proposition~A
arising in the context of locally coherent exact categories,
see~\cite[Theorems~7.1 and~8.3]{Plce}.

\subsection{{}} \label{introd-theorem-B}
 The next theorem and proposition, taken together, form a generalization
of Theorem~0(b).
 Theorem~B, which is the main claim, is rather easily deduced from
a result of \v St\!'ov\'\i\v cek and the present
author~\cite[Theorem~9.1]{PS6} (which, in turn, is a generalization
of~\cite[Theorem~4.7]{BCE}).
 Proposition~B, which is a supplementary comment on the theorem 
(explaining what the theorem says under some additional assumptions),
turns out to be more involved.

\begin{thmb}
 Let\/ $\sK$ be a Grothendieck category and\/ $\sS\subset\sK$ be
a class of objects.
 Let\/ $\sT\subset\sK$ be any class of objects of finite projective
dimension in\/~$\sK$ such that the union\/ $\sS\cup\sT$ contains
a set of generators for\/~$\sK$.
 Denote by\/ $\sC\subset\sK$ the closure of\/ $\sS\cup\sT$ under
coproducts, direct limits, extensions, and kernels of epimorphisms
in\/~$\sK$.
 Let\/ $\sB=\sS^{\perp_1}$ be the class of all objects
$B\in\sK$ such that\/ $\Ext^1_\sK(S,B)=0$ for all $S\in\sS$, and let\/
$\sD=\sC^{\perp_1}$ be the class of all objects $D\in\sK$ such that\/
$\Ext^1_\sK(C,D)=0$ for all $C\in\sC$.
 Then, for any acyclic complex of objects from\/ $\sD$ with the objects
of cocycles belonging to\/ $\sB$, the objects of cocycles actually
belong to\/~$\sD$.
\end{thmb}

\begin{propb}
 In the context of Theorem~B, if the class\/ $\sS\cup\sT$ consists of
(some) finitely presentable objects and is closed under extensions and
kernels of epimorphisms in\/ $\sK$, then\/ $\sC$ coincides with\/
$\varinjlim(\sS\cup\sT)$, the class consisting of all direct limits of
diagrams of objects from\/ $\sS\cup\sT$, indexed by directed posets.
\end{propb}

 One can see that under the assumptions of Proposition~B the class
$\sS\cup\sT$ has to consist of strongly finitely presentable
($\FP_\infty$) objects.
 Taking $\sT=\varnothing$ makes Theorem~0(b) a~particular case of
Theorem~B with Proposition~B (for $\sK=\Modr R$).

 Taking $\sS$ to be the class of all finitely presentable objects in
a locally finitely presentable abelian category $\sK$ and
$\sT=\varnothing$, one obtains the assertion~(4) or~(4$^\text{c}$)
on the list of Section~\ref{introd-seven-items}, essentially due
to \v St\!'ov\'\i\v cek~\cite[Corollary~5.5]{Sto}, as a particular case
of Theorem~B (for $\sK=\Modr R$).

 For another particular case of Theorem~B with Proposition~B arising
in the context of locally coherent exact categories,
see~\cite[Theorem~7.6]{Plce}.
 
\subsection{{}}
 The proofs of Theorems~0(b) and~B are presented in
Section~\ref{generalized-cotorsion-periodicity-secn}.
 Theorem~0(a) is proved in 
Section~\ref{generalized-fp-projective-periodicity-I-secn}.
 The proofs of Proposition~A(i) and Proposition~B are given in
Section~\ref{direct-limit-closures-of-FP-classes}.
 Proposition~A(ii) is proved in Section~\ref{pure-exact-structure-secn}.
 The possibilities and difficulties of extending Proposition~A
to higher cardinals~$\kappa$ are discussed in
Section~\ref{direct-limits-of-kappa-P-classes}.
 The proof of the main result, Theorem~A, is presented in
Section~\ref{generalized-fp-projective-periodicity-II-secn}.

 One comment on the style of the exposition may be in order.
 This paper is written with the intent to be at least partially
understandable to readers not necessarily feeling at ease with
advanced category-theoretic concepts.
 In order not to intimidate a reader mostly interested in
module-theoretic rather than category-theoretic applications,
category-theoretic terminology is introduced slowly and gradually
as the paper progresses from the less general results such as
Theorem~0 to the more general ones such as Theorem~B, Proposition~B,
Proposition~A, and Theorem~A.
 This order of exposition also allows us to make a better connection
with the preceding work in module theory, such as~\cite{AT}
and~\cite{GT}.

 Finally, let us make one terminological remark.
 In this paper, we generally refer to skeletally small classes as
``classes'' rather than ``sets''.
 So, the collection of all finitely presented modules over a given ring
or all $\kappa$\+presentable objects in a given Grothendieck category is
a class in our terminology.

\subsection*{Acknowledgement}
 I~want to thank Michal Hrbek, Silvana Bazzoni, Jan \v Saroch, and
 Jan Trlifaj for helpful discussions and comments.
 Long conversations with Jan \v St\!'ov\'\i\v cek were particularly
illuminating, and to him goes my special gratitude.
 I~also wish to thank an anonymous referee for helpful suggestions.
 The author was supported by the GA\v CR project 20-13778S and
the Institute of Mathematics, Czech Academy of Sciences
(research plan RVO:~67985840).

\Section{Generalized Fp-injective/Injective and Cotorsion Periodicity}
\label{generalized-cotorsion-periodicity-secn}

 In this section we prove Theorems~0(b) and~B.
 This is not difficult, given the preceding results
in~\cite[Theorem~4.7]{BCE} and~\cite[Theorem~9.1]{PS6}.
 The former theorem needs to be used together
with~\cite[Corollary~2.4]{AT}, and the latter one together
with~\cite[Lemmas~7.4 and~9.3]{PS6}.

 Let us formally introduce some notation and terminology which was
already used throughout the introduction.
 Given an abelian (or exact~\cite{Bueh}) category $\sK$ and a class of
objects $\sA\subset\sK$, one denotes by $\sA^{\perp_1}\subset\sK$
the class of all objects $X\in\sK$ such that $\Ext^1_\sK(A,X)=0$
for all $A\in\sA$.
 Dually, for a class of objects $\sB\subset\sK$, the notation
${}^{\perp_1}\sB\subset\sK$ stands for the class of all objects
$Y\in\sK$ such that $\Ext^1_\sK(Y,B)=0$ for all $B\in\sB$.
 Similarly, $\sA^{\perp_{\ge1}}\subset\sK$ is the class of all objects
$X\in\sK$ such that $\Ext^n_\sK(A,X)=0$ for all $A\in\sA$ and $n\ge1$.
 Dually, ${}^{\perp_{\ge1}}\sB\subset\sK$ is the class of all objects
$Y\in\sK$ such that $\Ext^n_\sK(Y,B)=0$ for all $B\in\sB$ and $n\ge1$.

 A class of objects $\sA\subset\sK$ is said to be \emph{generating}
(or \emph{a class of generators}) if every object of $\sK$ is
a quotient object of a coproduct of objects from~$\sA$.
 A class of objects $\sB\subset\sK$ is said to be \emph{cogenerating}
(or \emph{a class of cogenerators}) if every object of $\sK$ is
a subobject of a product of objects from~$\sB$.

 The previous definitions, as well as generally all category-theoretic
definitions in this paper, are transferred from abelian to exact
categories in the obvious way: all the mentions of ``subobjects'',
``quotients'', ``monomorphisms'', ``epimorphisms'', ``exact sequences'', 
etc., are understood to mean admissible monomorphisms, admissible
epimorphisms, admissible exact sequences, etc.

 A pair of classes of objects $(\sA,\sB)$ in $\sK$ is said to be
a \emph{cotorsion pair} if $\sA={}^{\perp_1}\sB$ and
$\sB=\sA^{\perp_1}$.
 Notice that, for any cotorsion pair $(\sA,\sB)$ in $\sK$, the class
$\sA\subset\sK$ is closed under coproducts (i.~e., those coproducts
that exist in~$\sK$), and the class $\sB\subset\sK$ is closed under
products (in the same sense)~\cite[Corollary 8.3]{CoFu},
\cite[Corollary A.2]{CoSt}.

 For any class of objects $\sS\subset\sK$, the pair of classes
$\sB=\sS^{\perp_1}$ and $\sA={}^{\perp_1}\sB$ is a cotorsion pair
in~$\sK$.
 The cotorsion pair $(\sA,\sB)$ obtained in this way is said to be
\emph{generated} by the class~$\sS$.
 Dually, for any class of objects $\sT\subset\sK$, the pair of classes
$\sA={}^{\perp_1}\sT$ and $\sB=\sA^{\perp_1}$ is also a cotorsion pair
in~$\sK$.
 The latter cotorsion pair $(\sA,\sB)$ is said to be \emph{cogenerated}
by the class~$\sT$.

 Let $(\sA,\sB)$ be a cotorsion pair in $\sK$ such that the class $\sA$
is generating and the class $\sB$ is cogenerating in\/~$\sK$.
 So every object of $\sK$ is a quotient object of an object from $\sA$
and a subobject of an object from~$\sB$.
 These conditions are satisfied automatically for any cotorsion pair
in an abelian category $\sK$ with enough projective and injective
objects (because all projective objects belong to $\sA$ and all
injective objects belong to~$\sB$).
 In particular, this applies to the module categories $\sK=\Modr R$.

 In the assumptions of the previous paragraph, the following conditions
are equivalent~\cite[Theorem~1.2.10]{GR}, \cite[Lemma~5.24]{GT},
\cite[Section~1]{BHP}, \cite[Lemma~7.1]{PS6}:
\begin{enumerate}
\item the class\/ $\sA$ is closed under kernels of epimorphisms
in\/~$\sK$;
\item the class\/ $\sB$ is closed under cokernels of monomorphisms
in\/~$\sK$;
\item $\Ext_\sK^2(A,B)=0$ for all $A\in\sA$ and $B\in\sB$;
\item $\Ext_\sK^n(A,B)=0$ for all $A\in\sA$, \,$B\in\sB$, and $n\ge1$.
\end{enumerate}
 A cotorsion pair $(\sA,\sB)$ satisfying conditions~(1\+-4) is said
to be \emph{hereditary}.

 Given a class of objects $\sL\subset\sK$, an object $M\in\sK$ is
said to be \emph{$\sL$\+periodic} if there exists a short exact
sequence $0\rarrow M\rarrow L\rarrow M\rarrow0$
\,\eqref{periodicity-sequence} in $\sK$ with $L\in\sL$.
 We recall that the notation $\varinjlim\sL\subset\sK$ stands for
the class of all direct limits in $\sK$ of diagrams of objects
from $\sL$ (indexed by directed posets).

 The following proposition is a generalization of~\cite[proof of
Proposition~7.6]{CH} and~\cite[Propositions~1 and~2]{EFI}.

\begin{prop} \label{chop-or-splice-prop}
 Let\/ $\sK$ be an abelian category and\/ $\sL\subset\sL'\subset\sM$
be three classes of objects in\/~$\sK$.
 Consider the following two properties:
\begin{enumerate}
\item In every acyclic complex $L^\bu$ in\/ $\sK$ with the terms
$L^n\in\sL$, $\,n\in\boZ$, and with the objects of cocycles of $L^\bu$
belonging to\/ $\sM$, the objects of cocycles belong to\/~$\sL'$.
\item In every short exact sequence $0\rarrow M\rarrow L\rarrow
M\rarrow0$ in\/ $\sK$ with the objects $L\in\sL$ and $M\in\sM$, one
has $M\in\sL'$.
\end{enumerate}
 In this setting, the implication (1)\/\,$\Longrightarrow$\,(2) holds
true.
 If countable coproducts exist and are exact in\/ $\sK$, the classes\/
$\sL$ and\/ $\sM$ are closed under countable coproducts in\/ $\sK$, and
the class\/ $\sL'$ is closed under direct summands in\/ $\sK$,
then the implication (2)\/\,$\Longrightarrow$\,(1) also holds.
 Dually, if countable products exist and are exact in\/ $\sK$,
the classes\/ $\sL$ and\/ $\sM$ are closed under countable products
in\/ $\sK$, and the class\/ $\sL'$ is closed under direct summands
in\/ $\sK$, then the implication (2)\/\,$\Longrightarrow$\,(1) holds. 
\end{prop}

\begin{proof}
 The implication (1)\,$\Longrightarrow$\,(2) is provable by
splicing up a doubly unbounded sequence of short exact sequences
$0\rarrow M\rarrow L\rarrow M\rarrow0$ and applying~(1) to
the resulting doubly unbounded complex.
 To prove the implication (2)\/\,$\Longrightarrow$\,(1), one needs
to chop up the complex $L^\bu$ into short exact sequence pieces
and apply~(2) to the infinite (co)product of the pieces.
 We refer to the proofs in~\cite[Proposition~7.6]{CH}
or~\cite[Propositions~1 and~2]{EFI} for the details.
\end{proof}

 A short exact sequence of right $R$\+modules $0\rarrow K\rarrow L
\rarrow M\rarrow0$ is said to be \emph{pure} if it remains exact
after taking the tensor product with any left $R$\+module.
 Equivalently, a short exact sequence $0\rarrow K\rarrow L\rarrow M
\rarrow0$ is pure if and only if it remains exact after applying
the functor $\Hom_R(S,{-})$ from any finitely presented right
$R$\+module~$S$ \,\cite[Definition~2.6 and Lemma~2.19]{GT}.
 If this is the case, the object $K$ is said to be a \emph{pure
subobject} of $L$, while the object $M$ is called a \emph{pure
epimorphic image} (or a \emph{pure quotient}) of~$M$.
 The \emph{pure exact structure} on $\Modr R$ is formed by the class
of all pure exact sequences.
 The projective objects of the exact category $\Modr R$ with the pure
exact structure are called \emph{pure-projective} $R$\+modules, and
the injective objects are called \emph{pure-injective}.

 An $R$\+module $S$ is said to be $\FP_n$ (where $n\ge0$ is
an integer) if it admits a fragment of projective resolution
$P_n\rarrow P_{n-1}\rarrow\dotsb\rarrow P_0\rarrow S\rarrow0$
with finitely generated projective modules~$P_i$.
 So a module is $\FP_0$ if and only if it is finitely generated, and
it is $\FP_1$ if and only if it is finitely presented.
 A module $S$ is said to be $\FP_\infty$ if it admits a resolution by
finitely generated projective modules; equivalently, this means that
$S$ is $\FP_n$ for all $n\ge0$.
 Modules of type $\FP_\infty$ are otherwise known as \emph{strongly
finitely presented}.

 A class of modules $\sS$ is said to be \emph{closed under syzygies}
if for every module $S\in\sS$ there exists a short exact sequence
$0\rarrow K\rarrow P\rarrow S\rarrow0$ with a projective module $P$
and $K\in\sS$.
 For any other short exact sequence $0\rarrow K'\rarrow P'\rarrow S
\rarrow0$ with a projective module $P'$, it then follows that
$K'\oplus P\simeq K\oplus P'$, which often implies that $K'\in\sS$
as well.
 Dually, a class of modules $\sT$ is \emph{closed under cosyzygies}
if for every module $T\in\sT$ there exists a short exact sequence
$0\rarrow T\rarrow J\rarrow L\rarrow0$ with an injective module $J$
and $L\in\sT$.
 For any other short exact sequence $0\rarrow T\rarrow J'\rarrow L'
\rarrow0$ with an injective module $J'$, it then follows that
$L'\oplus J\simeq L\oplus J'$, which often implies that $L'\in\sT$, too.

\begin{proof}[Proof of Theorem~0(b) from
Section~\ref{introd-theorem-zero}]
 Let us first prove the first assertion, then deduce the second one.
 In order to apply~\cite[Theorem~4.7]{BCE}, we need to show that
$(\sC,\sD)$ is a hereditary cotorsion pair in $\Modr R$.
 First of all, $(\sC,\sD)$ is indeed a cotorsion pair
by~\cite[Corollary~2.4]{AT} (see also~\cite[Corollary~8.42]{GT}).
 Alternatively,
Corollary~\ref{angeleri-trlifaj-cotorsion-pair-in-lfp-category} below
provides a more general result.

 To show that the cotorsion pair $(\sC,\sD)$ is hereditary, one can
argue as follows.
 The cotorsion pair $(\sA,\sB)$ generated by $\sS$ in $\Modr R$ is
hereditary, since the class $\sS$ is closed under
syzygies~\cite[Theorem~2.6]{CEG}, \cite[Corollary~5.25(a)]{GT},
\cite[Lemma~1.3]{BHP}, \cite[Lemma~7.1]{PS6}.
 Consequently, the class $\sB$ is closed under the cokernels of
monomorphisms, and in particular, under cosyzygies in $\Modr R$.
 By~\cite[Corollary~2.4]{AT} or~\cite[Corollary~8.42]{GT},
the cotorsion pair $(\sC,\sD)$ is cogenerated by the class of all
pure-injective modules belonging to~$\sB$.
 The class of all pure-injective modules is closed under cosyzygies
by~\cite[Lemma~6.20]{GT}, so the class of all pure-injective
modules belonging to $\sB$ is also closed under cosyzygies.
 Applying~\cite[Corollary~5.25(b)]{GT}, we conclude that the cotorsion
pair $(\sC,\sD)$ is hereditary.
 Alternatively, one can use Proposition~\ref{varinjlim-kernel-closed}
below, which is a more general result.

 We also need to know that the class $\sC$ is closed under pure
epimorphic images.
 This is~\cite[Proposition~2.1]{Len}, \cite[Section~4.1]{CB},
\cite[Theorem~2.3]{AT}, or~\cite[Theorem~8.40]{GT}.
 By the latter two references, we also have $\sA\subset\sC$, hence
$\sB\supset\sD$.

 Therefore, the result of~\cite[Theorem~4.7]{BCE} is applicable to
the cotorsion pair $(\sC,\sD)$, and it tells that the class
$\sC\cap{}^{\perp_1}\{M\}$ is closed under direct limits in $\Modr R$
for any $\sD$\+periodic module~$M$.
 Now, if $M\in\sB$, then the class $\sC\cap{}^{\perp_1}\{M\}$
contains~$\sA$.
 Thus $\sC=\varinjlim\sA\subset{}^{\perp_1}\{M\}$ and $M\in\sD$.

 To deduce the second assertion of Theorem~0(b) from the first one,
we apply Proposition~\ref{chop-or-splice-prop}\,(2)\,$\Rightarrow$\,(1)
to the category $\sK=\Modr R$ and the classes of objects
$\sL=\sL'=\sD$, \ $\sM=\sB$.
 Here we need to use the observations that countable products are
exact in $\Modr R$, the classes $\sD$ and $\sB$ are closed under
countable products, and the class $\sD$ is closed under direct
summands in $\Modr R$.
\end{proof}

 Let $\sK$ be an abelian category.
 A class of objects $\sC\subset\sK$ is said to be
\emph{self-generating}~\cite[Section~1]{BHP}, \cite[Section~7]{PS6}
if for any epimorphism $K\rarrow C$ in $\sK$ with $C\in\sC$ there
exists a morphism $C'\rarrow K$ in $\sK$ with $C'\in\sC$ such that
the composition $C'\rarrow K\rarrow C$ is an epimorphism in~$\sK$.
 A class of objects $\sC$ is said to be
\emph{self-resolving}~\cite[Section~9]{PS6} if it is self-generating
and closed under extensions and kernels of epimorphisms.

 Before proving Theorem~B as it is stated, let us explicitly
formulate and prove the following periodicity assertion.

\begin{thm} \label{theorem-B-as-periodicity-claim}
 Let\/ $\sK$ be a Grothendieck category and\/ $\sS\subset\sK$ be
a class of objects.
 Let\/ $\sT\subset\sK$ be any class of objects of finite projective
dimension in\/~$\sK$ such that the union\/ $\sS\cup\sT$ contains
a set of generators for\/~$\sK$.
 Denote by\/ $\sC\subset\sK$ the closure of\/ $\sS\cup\sT$ under
coproducts, direct limits, extensions, and kernels of epimorphisms
in\/~$\sK$.

 Put\/ $\sB=\sS^{\perp_1}$ and\/ $\sD=\sC^{\perp_1}\subset\sK$.
 Then, for any short exact sequence~\eqref{periodicity-sequence}
as in Section~\ref{introd-seven-items} with objects $L\in\sD$ and
$M\in\sB$, one has $M\in\sD$.
 In other words, any\/ $\sD$\+periodic object belonging to\/ $\sB$
actually belongs to\/~$\sD$.
\end{thm}

\begin{proof}
 The class of objects $\sC$ contains a set of generators for $\sK$ and
is closed under coproducts; hence, in particular, it is self-generating.
 The class $\sC$ is also closed under extensions and kernels of
epimorphisms in~$\sK$; so it is self-resolving.
 Finally, the class $\sC$ is closed under direct limits in $\sK$, and
the direct limits are exact in~$\sK$.
 Thus the assumptions of~\cite[Theorem~9.1]{PS6} are satisfied for
the class $\sC\subset\sK$, which tells that, for any $\sD$\+periodic
object $M\in\sK$, the class $\sC\cap{}^{\perp_1}\{M\}$ is closed
under direct limits in~$\sC$, or equivalently, in~$\sK$.

 By~\cite[Lemma~1.3]{BHP} or~\cite[Lemma~7.1]{PS6}, we have
$\Ext^n_\sK(C,D)=0$ for all objects $C\in\sC$, \ $D\in\sD$, and
integers $n\ge1$.
 By~\cite[Lemma~7.4]{PS6}, it follows that the class
$\sC\cap{}^{\perp_1}\{M\}$ contains all objects of the class $\sC$
which have finite projective dimension in~$\sK$.
 Thus $\sT\subset\sC\cap{}^{\perp_1}\{M\}$.
 If $M\in\sB$, then we also have $\sS\subset\sC\cap{}^{\perp_1}\{M\}$.

 On the other hand, \cite[Lemma~7.3 or Theorem~9.1]{PS6} also tells
that the class $\sC\cap{}^{\perp_1}\{M\}$ is closed under extensions
and kernels of admissible epimorphisms in the exact category $\sC$
(with the exact category structure inherited from the abelian exact
structure of~$\sK$).
 Since the class $\sC$ is closed under extensions and kernels of
epimorphisms in $\sK$, it follows that the class
$\sC\cap{}^{\perp_1}\{M\}$ is closed under extensions and
kernels of epimorphisms in~$\sK$.
 Finally, the class $\sC\cap{}^{\perp_1}\{M\}$ is closed under
coproducts in $\sK$, since it is closed under finite direct sums
and direct limits.

 We have shown that the class $\sC\cap{}^{\perp_1}\{M\}$
contains $\sS\cup\sT$ and is closed under extensions, kernels of
epimorphisms, coproducts, and direct limits in~$\sK$.
 Hence we can conclude that $\sC\cap{}^{\perp_1}\{M\}=\sC$, so
$\sC\subset{}^{\perp_1}\{M\}$ and $M\in\sD$.
\end{proof}

\begin{proof}[Proof of Theorem~B from Section~\ref{introd-theorem-B}]
 Proposition~\ref{chop-or-splice-prop}\,(2)\,$\Rightarrow$\,(1) as it
is stated above is \emph{not} applicable here, because the classes
$\sB$ and $\sD$ need not be closed under coproducts in $\sK$, while
countable products need not be exact in~$\sK$.
 So our argument is based on~\cite[Lemma~9.3]{PS6}.

 Let $D^\bu$ be an acyclic complex in $\sK$ with the terms $D^i\in\sD$
and the objects of cocycles $B^i\in\sB$.
 So we have short exact sequences $0\rarrow B^i\rarrow D^i\rarrow
B^{i+1}\rarrow0$ in~$\sK$.
 Taking the product of these short exact sequences over $i\in\boZ$,
we obtain a sequence
\begin{equation} \label{product-sequence}
 0\lrarrow\prod\nolimits_{i\in\boZ} B^i\lrarrow
 \prod\nolimits_{i\in\boZ}D^i\lrarrow\prod\nolimits_{i\in\boZ}B^i
 \lrarrow0.
\end{equation}

 In order to show that~\eqref{product-sequence} is exact, we
apply~\cite[Lemma~9.3]{PS6}.
 By assumption, the class $\sS\cup\sT$ contains a set of generators
of the Grothendieck category~$\sK$.
 So there exists a family of objects $(G_\xi)_{\xi\in\Xi}$ in
$\sS\cup\sT$ together with an epimorphism $G=\coprod_{\xi\in\Xi}G_\xi
\rarrow\prod_{i\in\boZ}B^i$ in~$\sK$.
 It remains to show that $\Ext^1_\sK(G,B^i)=0$ for every $i\in\boZ$.

 By~\cite[Corollary 8.3]{CoFu} or~\cite[Corollary A.2]{CoSt}, it
suffices to check that $\Ext^1_\sK(G_\xi,B^i)=0$ for every $i\in\boZ$
and $\xi\in\Xi$.
 There are two cases.
 If $G_\xi\in\sS$, then it remains to recall that
$B^i\in\sB=\sS^{\perp_1}$.
 If $G_\xi\in\sT$, then $G_\xi\in\sC$ and the projective dimension
of $G_\xi$ in $\sK$ is finite.
 From the short exact sequences $0\rarrow B^j\rarrow D^j\rarrow B^{j+1}
\rarrow0$ we get $\Ext^1_\sK(G_\xi,B^i)\simeq\Ext^2_\sK(G_\xi,B^{i-1})
\simeq\Ext^3_\sK(G_\xi,B^{i-2})\simeq\dotsb=0$, since
$\Ext^n_\sK(G_\xi,D^j)=0$ for all $j\in\boZ$ and $n\ge1$ as explained
in the proof of Theorem~\ref{theorem-B-as-periodicity-claim}
(cf.~\cite[proof of Proposition~9.4]{PS6}).
 So~\cite[Lemma~9.3]{PS6} tells that
the short sequence~\eqref{product-sequence} is exact.

 Applying~\cite[Corollary 8.3]{CoFu} or~\cite[Corollary A.2]{CoSt}
again, we see that both the classes $\sB$ and $\sD$ are closed under
infinite products in~$\sK$.
 Hence $\prod_{i\in\boZ}B^i\in\sB$ and $\prod_{i\in\boZ}D^i\in\sD$.
 So $\prod_{i\in\boZ}B^i$ is a $\sD$\+periodic object in $\sB$.
 By Theorem~\ref{theorem-B-as-periodicity-claim}, it follows that
$\prod_{i\in\boZ}B^i\in\sD$.
 Finally, the class $\sD$ is closed under direct summands in $\sK$,
hence $B^i\in\sD$ for all $i\in\boZ$.
\end{proof}

\Section{Generalized Flat/Projective and Fp-projective Periodicity~I}
\label{generalized-fp-projective-periodicity-I-secn}

 The aim of this section is to prove Theorem~0(a).
 It is restated below as
Theorem~\ref{theorem-0a-detailed-in-two-parts}(a) and
Corollary~\ref{theorem-0a-C-A-periodicity-cor}.
 The argument follows the ideas of the proof
of~\cite[Theorems~0.7--0.8 or Corollaries~4.7--4.9]{BHP}.
 The result is module-theoretic, but the proof has a category-theoretic
flavor in that the approach of~\cite{BHP} needs to be applied
\emph{within the class\/ $\sC$ viewed as an exact subcategory\/
$\sC\subset\Modr R$}.

 Let $\sK$ be an exact category (in Quillen's sense).
 We suggest the survey paper~\cite{Bueh} as a general reference
source on exact categories.
 The definition of a (\emph{hereditary}) \emph{cotorsion pair}
$(\sA,\sB)$ in $\sK$ was already given in the beginning of
Section~\ref{generalized-cotorsion-periodicity-secn}.
 The intersection of the two classes $\sA\cap\sB\subset\sK$ is called
the \emph{kernel} of a cotorsion pair $(\sA,\sB)$.
 Let us define the important concept of a \emph{complete}
cotorsion pair.

 A cotorsion pair $(\sA,\sB)$ in $\sK$ is said to be \emph{complete}
if for every object $K\in\sK$ there exist (admissible) short exact
sequences in $\sK$ of the form
\begin{gather}
 0\lrarrow B'\lrarrow A\lrarrow K\lrarrow0
 \label{spec-precover-sequence} \\
 0\lrarrow K\lrarrow B\lrarrow A'\lrarrow0
 \label{spec-preenvelope-sequence}
\end{gather}
with $A$, $A'\in\sA$ and $B$, $B'\in\sB$.
 The sequence~\eqref{spec-precover-sequence} is called
a \emph{special precover sequence}.
 The sequence~\eqref{spec-preenvelope-sequence} is called
a \emph{special preenvelope sequence}.
 Collectively, the sequences~(\ref{spec-precover-sequence}\+-%
\ref{spec-preenvelope-sequence}) are referred to as
the \emph{approximation sequences}.

 Let $\sE\subset\sK$ be a full subcategory closed under extensions.
 Then we endow $\sE$ with the exact category structure \emph{inherited
from} the exact category structure of~$\sK$.
 The short exact sequences in the inherited exact structure on $\sE$ are
the short exact sequences in $\sK$ with the terms belonging to~$\sE$.

\begin{lem} \label{kernel-is-injectives-in-left-class}
 Let $(\sC,\sD)$ be a complete cotorsion pair in an exact
category~$\sK$.
 Then the exact category\/ $\sC$ (with the exact structure inherited
from\/~$\sK$) has enough injective objects.
 The class of all injective objects in\/ $\sC$ is precisely
the kernel\/ $\sC\cap\sD$ of the cotorsion pair $(\sC,\sD)$.
 Dually, the exact category\/ $\sD$ has enough projective objects,
and the kernel\/ $\sC\cap\sD$ is precisely the class of all
projectives in\/~$\sD$.
\end{lem}

\begin{proof}
 The proof is left to the reader.
\end{proof}

 Let $\sK$ be an exact category and $\sE\subset\sK$ be a full
subcategory closed under extensions, endowed with the inherited
exact category structure.
 Let $(\sA,\sB)$ be a complete cotorsion pair in~$\sK$.
 We will say that the cotorsion pair $(\sA,\sB)$ \emph{restricts to}
(\emph{a complete cotorsion pair in}) the exact subcategory $\sE$ if
the pair of classes ($\sE\cap\sA$, $\sE\cap\sB$) is a complete
cotorsion pair in~$\sE$.

\begin{lem} \label{restricted-cotorsion-pair}
 Let $(\sA,\sB)$ be a complete cotorsion pair in an exact category\/
$\sK$, and let\/ $\sE\subset\sK$ be a full subcategory closed under
extensions and kernels of admissible epimorphisms.
 Assume that\/ $\sA\subset\sE$.
 Then \par
\textup{(a)} the cotorsion pair $(\sA,\sB)$ restricts to\/ $\sE$,
so $(\sA$, $\sE\cap\sB)$ is a complete cotorsion pair in\/~$\sE$; \par
\textup{(b)} if the cotorsion pair $(\sA,\sB)$ is hereditary in\/
$\sK$, then the restricted cotorsion pair $(\sA$, $\sE\cap\sB)$ is
hereditary in\/~$\sE$.
\end{lem}

\begin{proof}
 This is fairly standard and easy to prove.
 The details can be found, e.~g., in~\cite[Lemmas~1.5(a) and~1.6]{Pal}.
\end{proof}

 Given an additive category $\sK$, we denote by $\bC(\sK)$ the additive
category of complexes in~$\sK$ (with the usual morphisms of complexes)
and by $\bH(\sK)$ the triangulated homotopy category of complexes
in~$\sK$.
 So the morphisms in $\bH(\sK)$ are the cochain homotopy classes of
morphisms in $\bC(\sK)$.
 When $\sK$ is an exact category, the category $\bC(\sK)$ is endowed
with the exact category structure in which a short sequence of
complexes is exact if and only if it is exact at every degree.
 We denote by $K^\bu\longmapsto K^\bu[n]$ the functor of grading shift
on the complexes; so $K^\bu[n]^i=K^{n+i}$ for all $n$, $i\in\boZ$.

\begin{lem} \label{ext-homotopy-hom-lemma}
 Let\/ $\sK$ be an exact category, and let $A^\bu$ and $B^\bu$ be two
complexes in\/~$\sK$.
 Assume that\/ $\Ext^1_\sK(A^n,B^n)=0$ for every $n\in\boZ$.
 Then there is a natural isomorphism of abelian groups
$$
 \Ext^1_{\bC(\sK)}(A^\bu,B^\bu)\simeq\Hom_{\bH(\sK)}(A^\bu,B^\bu[1]).
$$
\end{lem}

\begin{proof}
 This is also standard and well-known.
 More generally, for any two complexes $A^\bu$ and $B^\bu$ in $\sK$,
the subgroup of \emph{termwise split} extensions $0\rarrow B^\bu\rarrow
C^\bu\rarrow A^\bu\rarrow0$ in $\Ext^1_\sK(A^\bu,B^\bu)$ is naturally
isomorphic to the group of morphisms $A^\bu\rarrow B^\bu[1]$ in
the homotopy category~$\bH(\sK)$.
 Stated in this form, the assertion essentially does not depend on
the exact structure on $\sK$ and is applicable to any additive category.
 We refer to~\cite[Lemma~1.6]{BHP} for the details (which are the same
in the general case as in the case of an abelian category $\sK$
discussed in~\cite{BHP}).
\end{proof}

 At this point, let us specialize our discussion to Grothendieck
abelian categories~$\sK$.
 Let $F\in\sK$ be an object and $\alpha$~be an ordinal.
 A family of subobjects $(F_\beta\subset F)_{0\le\beta\le\alpha}$ is
said to be an \emph{$\alpha$\+indexed filtration} on $F$ if
the following conditions are satisfied:
\begin{itemize}
\item $F_0=0$ and $F_\alpha=F$;
\item $F_\gamma\subset F_\beta$ for all $0\le\gamma\le\beta\le\alpha$;
\item $F_\beta=\bigcup_{\gamma<\beta}F_\gamma$ for all limit ordinals
$\beta\le\alpha$.
\end{itemize}

 An object $F\in\sK$ endowed with an ordinal-indexed filtration
$(F_\beta)_{0\le\beta\le\alpha}$ is said to be \emph{filtered by}
the quotient objects $S_\beta=F_{\beta+1}/F_\beta$, \
$0\le\beta<\alpha$.
 In an alternative terminology, the object $F$ is called
a \emph{transfinitely iterated extension} (\emph{in the sense of
the direct limit}) of the objects $(S_\beta)_{0\le\beta<\alpha}$.

 Given a class of objects $\sS\subset\sK$, the class of all objects
in $\sK$ filtered by (objects isomorphic to) objects from $\sS$
is denoted by $\Fil(\sS)\subset\sK$.
 A class of objects $\sF\subset\sK$ is said to be \emph{deconstructible}
if there exists a \emph{set} of objects $\sS\subset\sK$ such that
$\sF=\Fil(\sS)$.
 It is easy to see that any deconstructible class (in the sense of
this definition) is closed under transfinitely iterated extensions.

 The following result is known as
the \emph{Eklof lemma}~\cite[Lemma~1]{ET}, \cite[Lemma~6.2]{GT}.

\begin{lem} \label{eklof-lemma}
 For any class of objects\/ $\sB\subset\sK$, the class\/
${}^{\perp_1}\sB$ is closed under transfinitely iterated extensions.
 In other words, $\Fil({}^{\perp_1}\sB)={}^{\perp_1}\sB$.
\end{lem}

\begin{proof}
 This assertion, properly understood (as per the definitions
in Section~\ref{exact-grothendieck-secn} below), holds
in any exact category~$\sK$.
 See the references in~\cite[Lemma~1.1]{BHP}.
 The general formulation can be also found in~\cite[Lemma~7.5]{PS6}.
\end{proof}

 The next theorem goes back to Eklof and Trlifaj~\cite[Theorems~2
and~10]{ET}, \cite[Theorem~6.11 and Corollary~6.14]{GT}.
 For any class of objects $\sF\subset\sK$, we denote by $\sF^\oplus
\subset\sK$ the class of all direct summands of objects from $\sF$
in~$\sK$.

\begin{thm} \label{eklof-trlifaj-theorem}
 Let\/ $\sK$ be a Grothendieck category and\/ $(\sA,\sB)$ be
the cotorsion pair generated by a \emph{set} of objects\/
$\sS\subset\sK$.  Then \par
\textup{(a)} If the class\/ $\sA$ is generating in\/ $\sK$, then
the cotorsion pair $(\sA,\sB)$ is complete. \par
\textup{(b)} If the class\/ $\Fil(\sS)$ is generating in\/ $\sK$,
then\/ $\sA=\Fil(\sS)^\oplus$.
\end{thm}

\begin{proof}
 This result, properly stated, holds in any locally presentable
abelian category~$\sK$.
 See~\cite[Theorem~1.2]{BHP} for a discussion with references,
and Theorem~\ref{eklof-trlifaj-efficient-exact} below for
a version for efficient exact categories.
\end{proof}

 We refer to the book~\cite[Definition~1.9 and Theorem~1.11]{AR} for
the definition of a \emph{locally finitely presentable} category.
 Any locally finitely presentable abelian category is
Grothendieck~\cite[Proposition~1.59]{AR}.
 We will have a detailed discussion of such categories below in
Section~\ref{direct-limit-closures-of-FP-classes}, where several
further references are suggested.
 The abelian category of modules $\Modr R$ is locally finitely
presentable for any ring~$R$.

 The following result of \v St\!'ov\'\i\v cek was already used in
a similar way in the paper~\cite{BHP}, where it is stated
as~\cite[Lemma~3.4]{BHP}.

\begin{prop} \label{stovicek-countable-hill-prop}
 Let\/ $\sK$ be a locally finitely presentable abelian category and\/
$\sS\subset\sK$ be a class of finitely presentable objects closed
under extensions in\/~$\sK$.
 Let $A^\bu$ be a complex in\/ $\sK$ whose terms are\/ $\sS$\+filtered
objects.
 Then the complex $A^\bu$, viewed as an object of the abelian category
of complexes\/ $\bC(\sK)$, is filtered by bounded below complexes
of objects from\/~$\sS$.
\end{prop}

\begin{proof}
 This is the particular case of~\cite[(proof of) Proposition~4.3]{Sto0}
for the countable cardinal $\kappa=\aleph_0$.
 The argument is based on the Hill lemma (\cite[Theorem~2.1]{Sto0}
or~\cite[Theorem~7.10]{GT}).
\end{proof}

 In addition to the abelian exact structure on the module category
$\sK=\Modr R$, we are interested in the pure exact structure.
 The definition of the pure exact structure on $\Modr R$ was already
given in Section~\ref{generalized-cotorsion-periodicity-secn}.
 A complex in $\Modr R$ is said to be \emph{pure acyclic} (or
\emph{pure exact}) if it is acyclic in the pure exact structure,
i.~e., can be obtained by splicing pure short exact sequences.
 The following result due to Neeman~\cite{Neem} and
\v St\!'ov\'\i\v cek~\cite{Sto} is a stronger version of the pure
pure-projective periodicity theorem (item~(2) or~(2$^\text{c}$) on
the list of Section~\ref{introd-seven-items}).

\begin{thm} \label{neeman-stovicek-pure-projective-pure-acyclic}
 Let $R$ be an associative ring.
 Let $P^\bu$ be a complex of pure-projective $R$\+modules, and let
$X^\bu$ be a pure acyclic complex of $R$\+modules.
 Then any morphism of complexes of $R$\+modules $P^\bu\rarrow X^\bu$
is homotopic to zero.
\end{thm}

\begin{proof}
 This was first stated in~\cite[Theorem~5.4]{Sto} based
on~\cite[Theorem~8.6]{Neem}.
 We refer to the paper~\cite[Theorem~1.1]{BCE} for a generalization,
and to~\cite[Section~0.2 and proof of Theorem~4.3]{BHP} for
a discussion with some details.
\end{proof}

 Now let $\sS$ be a class of finitely presented $R$\+modules closed
under finite direct sums and containing the free $R$\+module~$R$.
 Put $\sC=\varinjlim\sS\subset\Modr R$.

\begin{lem} \label{restricted-pure-exact-structure}
 The full subcategory\/ $\sC=\varinjlim\sS$ is closed under pure
extensions (as well as pure submodules and pure epimorphic images)
in\/ $\Modr R$.
 In the exact category structure on\/ $\sC$ inherited from the pure
exact structure on\/ $\Modr R$, a short sequence\/ $0\rarrow K\rarrow L
\rarrow M\rarrow0$ is exact if and only if the short sequence of
abelian groups\/ $0\rarrow\Hom_R(S,K)\rarrow\Hom_R(S,L)\rarrow
\Hom_R(S,M)\rarrow0$ is exact for every module $S\in\sS$.
\end{lem}

\begin{proof}
 The first assertion is the result of~\cite[Proposition~2.2]{Len}.
 The second assertion claims that a short sequence $0\rarrow K
\rarrow L\rarrow M\rarrow0$ with $K$, $L$, $M\in\sC$ is pure exact
in $\Modr R$ if and only if the functor $\Hom_R(S,{-})$ takes it
to a short exact sequence for every $S\in\sS$.
 The point is that any morphism $T\rarrow M$ from a finitely
presented $R$\+module $T$ into the module $M\in\sC=\varinjlim\sS$
factorizes through some module $S\in\sS$.
 So if every morphism $S\rarrow M$ lifts to a morphism $S\rarrow L$,
then also every morphism $T\rarrow M$ lifts to a morphism $T\rarrow L$.
\end{proof}

 Now we can formulate and prove the main results of the section
(though we will need yet another lemma in between).

\begin{thm} \label{theorem-0a-detailed-in-two-parts}
 Let $R$ be a ring and\/ $\sS$ be a class of finitely presented
$R$\+modules, containing the free $R$\+module $R$ and closed under
extensions and kernels of epimorphisms.
 Put\/ $\sB=\sS^{\perp_1}$, \ $\sA={}^{\perp_1}\sB$, \
$\sC=\varinjlim\sS$, and\/ $\sD=\sC^{\perp_1}\subset\Modr R$.
 Then \par
\textup{(a)} in any acyclic complex of modules from\/ $\sA$ with
the modules of cocycles belonging to\/ $\sC$, the modules of
cocycles actually belong to\/~$\sA$; \par
\textup{(b)} let $A^\bu$ be a complex in\/ $\Modr R$ whose terms
belong to\/ $\sA$, and let $X^\bu$ be an acyclic complex in\/
$\Modr R$ whose terms belong to\/ $\sB\cap\sC$ and the modules of
cocycles also belong to\/~$\sB\cap\sC$.
 Then any morphism of complexes of modules $A^\bu\rarrow X^\bu$ is
homotopic to zero.
\end{thm}

 The following lemma tells that modules from the class $\sB\cap\sC$
are ``absolutely pure within the exact category~$\sC$''.

\begin{lem} \label{B-is-absolutely-pure-in-C}
 In the notation of Theorem~0 or
Theorem~\ref{theorem-0a-detailed-in-two-parts}, let\/
$0\rarrow B\rarrow L\rarrow C\rarrow0$ be a short exact sequence
in\/ $\Modr R$ with the terms $B$, $L$, $C\in\sC$.
 Assume that the module $B$ belongs to the class\/ $\sB\cap\sC$.
 Then the short exact sequence\/ $0\rarrow B\rarrow L\rarrow C
\rarrow0$ is pure in\/ $\Modr R$.
\end{lem}

\begin{proof}
 It is only important that $C\in\sC$ and $B\in\sB$.
 By Lemma~\ref{restricted-pure-exact-structure}, it suffices to
check that any morphism $S\rarrow C$ with $S\in\sS$ lifts to
a morphism $S\rarrow L$.
 This holds because $B\in\sB=\sS^{\perp_1}\subset\Modr R$.
\end{proof}

\begin{proof}[Proof of
Theorem~\ref{theorem-0a-detailed-in-two-parts}(b)]
 The argument follows the ideas of the proof of~\cite[Theorem~4.2]{BHP},
with suitable modifications.
 By the Eklof--Trlifaj theorem (Theorem~\ref{eklof-trlifaj-theorem}(b)),
we have $\sA=\Fil(\sS)^\oplus$.
 Without loss of generality we can assume that the terms of the complex
$A^\bu$ belong to $\Fil(\sS)$.
 Then, by Proposition~\ref{stovicek-countable-hill-prop} (applied in
the case of the module category $\sK=\Modr R$), the complex $A^\bu$ is 
filtered by (bounded below) complexes with the terms belonging to~$\sS$.

 By Lemma~\ref{ext-homotopy-hom-lemma}, for any complex $A^\bu$ with
the terms in $\sA$ and any complex $B^\bu$ with the terms in $\sB$
we have an isomorphism of abelian groups
$$
 \Ext^1_{\bC(\Modr R)}(A^\bu,B^\bu[-1])\simeq
 \Hom_{\bH(\Modr R)}(A^\bu,B^\bu).
$$
 So, instead of showing that $\Hom_{\bH(\Modr R)}(A^\bu,X^\bu)=0$
as desired in the theorem, it suffices to prove that
$\Ext^1_{\bC(\Modr R)}(A^\bu,X^\bu[-1])=0$.
 In view of the Eklof lemma (Lemma~\ref{eklof-lemma}) applied in
the abelian category $\sK=\bC(\Modr R)$, the question reduces to
showing that $\Ext^1_{\bC(\Modr R)}(S^\bu,X^\bu[-1])=0$ for
any complex $S^\bu$ with the terms belonging to $\sS$ and any
complex $X^\bu$ as in the theorem.
 Using Lemma~\ref{ext-homotopy-hom-lemma} again, we conclude that it
suffices to show that any morphism of complexes $S^\bu\rarrow X^\bu$
is homotopic to zero.

 Finally, we observe that all finitely presented $R$\+modules are
pure-projective (by the definitions), while any acyclic complex of
modules with the modules of cocycles in $\sB\cap\sC$ is pure acyclic
(by Lemma~\ref{B-is-absolutely-pure-in-C}).
 Thus any morphism of complexes $S^\bu\rarrow X^\bu$ is homotopic to
zero by the Neeman--\v St\!'ov\'\i\v cek theorem
(Theorem~\ref{neeman-stovicek-pure-projective-pure-acyclic}).
\end{proof}

\begin{proof}[Proof of
Theorem~\ref{theorem-0a-detailed-in-two-parts}(a)]
 It is clear that in the assumptions of the theorem all the modules
from $\sS$ have to be strongly finitely presented ($\FP_\infty$).
 Thus~\cite[Theorem~2.3 and Corollary~2.4]{AT}
or~\cite[Theorem~8.40, Corollary~8.42, and Theorem~6.19]{GT} are
applicable, telling that $(\sC,\sD)$ is a complete cotorsion
pair in $\Modr R$.
 The cotorsion pair $(\sA,\sB)$ is complete in $\Modr R$ by
the Eklof--Trlifaj theorem (Theorem~\ref{eklof-trlifaj-theorem}(a)).

 Both the cotorsion pairs $(\sA,\sB)$ and $(\sC,\sD)$ are hereditary,
as it was explained in the proof of Theorem~0(b) in
Section~\ref{generalized-cotorsion-periodicity-secn}.
 Applying Lemma~\ref{restricted-cotorsion-pair} to the abelian
category $\sK=\Modr R$ and the full subcategory $\sE=\sC$, we conclude
that ($\sA$, $\sB\cap\sC$) is a hereditary complete cotorsion pair
in the exact category~$\sC$.
 Lemma~\ref{kernel-is-injectives-in-left-class} tells that there are
enough injective objects in the exact category $\sC$, and the class
of such injective objects coincides with the intersection $\sC\cap\sD$.

 Given two complexes of right $R$\+modules $A^\bu$ and $B^\bu$, we
denote by $\Hom_R(A^\bu,B^\bu)$ the direct product totalization of
the bicomplex of Hom groups $\Hom_R(A^i,B^j)$.
 In particular, this notation applies if $B^\bu=B$ is just a single
$R$\+module, which is then considered as a complex of $R$\+modules
concentrated in the cohomological degree~$0$.

 Let $A^\bu$ be an acyclic complex of modules from~$\sA$.
 Then one can easily see that the modules of cocycles of $A^\bu$
belong to $\sA$ if and only if the complex of abelian groups
$\Hom_R(A^\bu,B)$ is acyclic for any module $B\in\sB$.
 This holds because $(\sA,\sB)$ is a cotorsion pair in $\Modr R$
(so $\sA={}^{\perp_1}\sB$).
 We will use a version of this observation made within the exact
category~$\sC$.

 So let $A^\bu$ be an acyclic complex of modules from $\sA$ with
the modules of cocycles belonging to~$\sC$.
 Then we observe that the modules of cocycles of $A^\bu$ belong
to $\sA$ if and only if the complex $\Hom_R(A^\bu,B)$ is acyclic for
any module $B\in\sB\cap\sC$.
 This holds because ($\sA$, $\sB\cap\sC$) is a cotorsion pair in $\sC$,
so $\sA=\sC\cap{}^{\perp_1}(\sB\cap\sC)$.

 Now let $D^\bu$ be an injective resolution of the object $B$ in
the exact category~$\sC$.
 So we have $B\in\sB\cap\sC$ by assumption, and $0\rarrow B\rarrow
D^0\rarrow D^1\rarrow D^2\rarrow\dotsb$ is an acyclic complex in
$\Modr R$ with the modules $D^n\in\sC\cap\sD$ and the modules of
cocycles belonging to~$\sC$.
 We observe that the modules of cocycles of the complex $D^\bu$ actually
belong to $\sB\cap\sC$, because $\sD\subset\sB$ and the class $\sB$
is closed under cokernels of monomorphisms.
 Essentially, this is a restatement of the claim that the cotorsion
pair $(\sA,\sB)$ is hereditary in $\Modr R$, or more specifically,
that the cotorsion pair ($\sA$, $\sB\cap\sC$) is hereditary in~$\sC$.

 Denote by $X^\bu$ the acyclic complex $(B\to D^\bu)$.
 Then the complex of abelian groups $\Hom_R(A^\bu,X^\bu)$ is acyclic
by Theorem~\ref{theorem-0a-detailed-in-two-parts}(b) (which we have
proved above).
 This holds because $A^\bu$ is a complex with the terms in $\sA$,
while $X^\bu$ is an acyclic complex with the terms in $\sB\cap\sC$
and the modules of cocycles in $\sB\cap\sC$.

 On the other hand, the complex of abelian groups $\Hom_R(A^\bu,D^\bu)$
is acyclic as well.
 This holds quite generally for any acyclic complex $A^\bu$ and
any bounded below complex of injective objects $D^\bu$ in any exact
category~$\sC$.
 Notice that in the situation at hand the complex of modules $A^\bu$
is acyclic in the exact category $\sC$, as its modules of cocycles
belong to $\sC$ by assumption.

 Since both the complexes $\Hom_R(A^\bu,X^\bu)$ and
$\Hom_R(A^\bu,D^\bu)$ are acyclic, and the complex $X^\bu$ has
the form $X^\bu=(B\to D^\bu)$, we can finally conclude that
the complex of abelian groups $\Hom_R(A^\bu,B)$ is acyclic.
\end{proof}

\begin{cor} \label{theorem-0a-C-A-periodicity-cor}
 Let $R$ be a ring and\/ $\sS$ be a class of finitely presented
$R$\+modules, containing the free $R$\+module $R$ and closed under
extensions and the kernels of epimorphisms.
 Put\/ $\sA={}^{\perp_1}(\sS^{\perp_1})$ and\/ $\sC=\varinjlim\sS$.
 Then, for any short exact sequence~\eqref{periodicity-sequence}
as in Section~\ref{introd-seven-items} with modules $L\in\sA$
and $M\in\sC$, one has $M\in\sA$.
 In other words, any\/ $\sA$\+periodic module belonging to\/ $\sC$
actually belongs to\/~$\sA$.
\end{cor}

\begin{proof}
 Follows from Theorem~\ref{theorem-0a-detailed-in-two-parts}(a) by
Proposition~\ref{chop-or-splice-prop}\,(1)\,$\Rightarrow$\,(2) applied
to the category $\sK=\Modr R$ and the classes of objects $\sL=\sL'=\sA$,
\ $\sM=\sC$.
\end{proof}

\begin{proof}[Proof of Theorem~0(a) from
Section~\ref{introd-theorem-zero}]
 The first assertion of Theorem~0(a) is provided by
Corollary~\ref{theorem-0a-C-A-periodicity-cor}, and the second one
by Theorem~\ref{theorem-0a-detailed-in-two-parts}(a).
\end{proof}

\Section{Direct Limit Closures of Classes of Finitely Presentables}
\label{direct-limit-closures-of-FP-classes}

 The aim of this section is to prove Propositions~A(i) and~B.
 They are restated below as
Propositions~\ref{FP2-varinjlim-extension-closed}
and~\ref{varinjlim-kernel-closed}.

 In this section we work with \emph{locally finitely presentable}
abelian categories.
 We suggest the book~\cite{AR} as a general reference source on
nonadditive locally (finitely) presentable and (finitely) accessible
categories.

 The definitions of a finitely presentable object and a locally finitely
presentable category can be found in~\cite[Definitions~1.1 and~1.9, and
Theorem~1.11]{AR} (it is helpful to keep in mind that in abelian
categories the notions of a generator and a strong generator coincide).
 All locally finitely presentable abelian categories have exact direct
limit functors, so they are Grothendieck~\cite[Proposition~1.59]{AR}.
 The abelian category of modules over an arbitrary ring $\sK=\Modr R$
is an important example of a locally finitely presentable abelian
category.

 Finitely accessible categories~\cite[Definition~2.1 and
Remark~2.2(1)]{AR} form a wider class than the locally finitely
presentable ones.
 The theory of finitely accessible additive categories goes back to
the paper~\cite[Section~2]{Len} (where they were not defined yet).
 Subsequently they were studied in the papers~\cite{CB,Kra} under
the name of ``locally finitely presented'' additive categories.
 All coproducts exist in any finitely accessible additive category.

 We suggest~\cite[Sections~8.1\+-8.2]{PS5} as an additional reference
source on locally finitely presentable abelian categories.
 Our proof of Proposition~A(i)
(Proposition~\ref{FP2-varinjlim-extension-closed} below) is a slight
generalization of the argument in~\cite[Proposition~8.4]{PS5}.

 \emph{Locally finitely generated} categories also form a wider class
than the locally finitely presentable ones.
 We refer to~\cite[Section~1.E]{AR} for a general discussion of
locally generated (nonadditive) categories and
to~\cite[Corollary~9.6]{PS1} for a very general form of the assertion
that any locally finitely generated abelian category is Grothendieck.
 A good reference source on locally finitely generated Grothendieck
categories and finitely generated/finitely presentable objects in
them is~\cite[\S V.3]{Sten}.

 The following definitions are very general.
 Let $\sK$ be a category with direct limits.
 An object $S\in\sK$ is said to be \emph{finitely presentable} if
the functor $\Hom_\sK(S,{-})\:\sK\rarrow\Sets$ preserves direct limits.
 An object $S$ is said to be \emph{finitely generated} if the same
functor preserves the direct limits of diagrams of monomorphisms.

 An abelian category $\sK$ with set-indexed coproducts is said to be
\emph{locally finitely generated} if it has a set of generators
consisting of finitely generated objects.
 In particular, the category $\sK$ is \emph{locally finitely
presentable} if it has a set of generators consisting of finitely
presentable objects.

 Given a finitely accessible additive category $\sK$, we denote
by $\sK_\fp\subset\sK$ the full subcategory of finitely presentable
objects in~$\sK$.
 In any locally finitely presentable abelian category $\sK$,
the full subcategory $\sK_\fp$ is closed under
cokernels~\cite[Proposition~1.3]{AR} and
extensions~\cite[Lemma~8.1]{PS5}.
 Similarly, the full subcategory of finitely generated objects in
a locally finitely generated abelian category $\sK$ is closed under
extensions and quotients~\cite[Lemma~V.3.1 and Proposition~V.3.2]{Sten}.

\begin{prop} \label{varinjlim-varinjlim-closed}
 Let\/ $\sK$ be a finitely accessible additive category and\/
$\sS\subset\sK_\fp$ be a class of finitely presentable objects closed
under finite direct sums.
 Then the class of objects\/ $\varinjlim\sS\subset\sK$ is closed under
coproducts and direct limits in\/~$\sK$.
 An object $L\in\sK$ belongs to\/ $\varinjlim\sS$ if and only if, for
any object $T\in\sK_\fp$, any morphism $T\rarrow L$ in\/ $\sK$
factorizes through an object from\/~$\sS$. 
\end{prop}

\begin{proof}
 This is~\cite[Proposition~2.1]{Len}, \cite[Section~4.1]{CB},
or~\cite[Proposition~5.11]{Kra}.
\end{proof}

\begin{cor} \label{kernel-of-coproduct-onto-direct-limit-cor}
 Let\/ $\sK$ be a locally finitely presentable abelian category
and\/ $\sS\subset\sK_\fp$ be a class of finitely presentable objects
closed under finite direct sums.
 Let $(H_i)_{i\in I}$ be a direct system of objects $H_i\in
\varinjlim\sS$, indexed by a directed poset~$I$.
 Then the kernel of the natural epimorphism
\begin{equation} \label{coproduct-onto-direct-limit-epimorphism}
 \coprod\nolimits_{i\in I} H_i\lrarrow
 \varinjlim\nolimits_{i\in I}H_i
\end{equation}
belongs to\/ $\varinjlim\sS$.
\end{cor}

\begin{proof}
 One can argue from purity considerations, observing that
the epimorphism~\eqref{coproduct-onto-direct-limit-epimorphism} is pure
(in the sense of the definition from~\cite[Section~3]{CB},
\cite[Section~4]{Sto} reproduced below in
Section~\ref{pure-exact-structure-secn}) and the class $\varinjlim\sS$
is closed under pure subobjects by the category-theoretic version
of~\cite[Proposition~2.2]{Len}.
 Alternatively, one can notice that the kernel
of~\eqref{coproduct-onto-direct-limit-epimorphism}
is a direct limit of coproducts of copies of the objects $H_i$,
following~\cite[proof of Proposition~4.1]{BPS}; then it remains
to refer to Proposition~\ref{varinjlim-varinjlim-closed}.
\end{proof}

 The following definitions appeared in the papers~\cite{Gil,BGP}.
 Let $\sK$ be a Grothendieck category and $n\ge1$ be an integer.
 An object $S\in\sK$ is said to be \emph{of type\/ $\FP_n$} if
the functors $\Ext_\sK^i(S,{-})\:\sK\rarrow\Ab$ preserve direct
limits for all $0\le i\le n-1$.
 So the objects of type $\FP_1$ are, by the definition, the finitely
presentable ones, while the objects of type $\FP_n$ for $n\ge2$ form
more narrow classes.

 An object $S\in\sK$ is said to be \emph{of type\/ $\FP_0$} if it is
finitely generated.
 An object $S$ is said to be \emph{of type\/ $\FP_\infty$} if it is
of type $\FP_n$ for every $n\ge0$, that is, in other words,
the functors $\Ext_\sK^i(S,{-})\:\sK\rarrow\Ab$ preserve direct
limits for all $i\ge0$.

 We use the term \emph{strongly finitely presentable object} as
a synonym for ``type $\FP_\infty$''.
 In the case of the module category $\sK=\Modr R$, these definitions
are equivalent to the ones from
Section~\ref{generalized-cotorsion-periodicity-secn}
(see~\cite[Proposition~3.10 of the published version
or Corollary~2.14 of the \texttt{arXiv} version]{BGP}).

 Closure properties of the classes of objects of type $\FP_n$ and
$\FP_\infty$ in locally finitely presentable abelian categories $\sK$
are listed in~\cite[Corollary~3.3]{Gil} and~\cite[Proposition~3.7 
of the published version or Proposition~2.8 of the \texttt{arXiv}
version]{BGP}.
 In particular, \cite[Proposition~3.7(1)]{BGP} tells that the class
of all objects of type $\FP_n$ is closed under extensions in~$\sK$.

\begin{prop} \label{FP2-varinjlim-extension-closed}
 Let\/ $\sK$ be a locally finitely presentable abelian category
and\/ $\sS$ be a class of (some) objects of type\/ $\FP_2$ closed under
extensions in\/~$\sK$.
 Then the class of objects\/ $\varinjlim\sS$ is also closed under
extensions in\/~$\sK$.
\end{prop}

\begin{proof}
 We follow the proof of~\cite[Proposition~8.4]{PS5}.
 Given an abelian category $\sK$ and two classes of objects $\sX$,
$\sY\subset\sK$, denote by $\sX*\sY$ the class of all objects $Z\in\sK$
for which there exists a short exact sequence $0\rarrow X\rarrow Z
\rarrow Y\rarrow0$ in $\sK$ with $X\in\sX$ and $Y\in\sY$.
 In the situation at hand, we need to prove that $\varinjlim\sS*
\varinjlim\sS\subset\varinjlim\sS$.
 For this purpose, we claim that the three inclusions
\begin{equation} \label{three-inclusions}
 \varinjlim\sT*\varinjlim\sT\subset\varinjlim(\varinjlim\sT*\sT)
 \subset\varinjlim\varinjlim(\sT*\sT)\subset\varinjlim(\sT*\sT)
\end{equation}
hold for any class of $\FP_2$ objects $\sT$ closed under finite
direct sums in~$\sK$.

 More generally, it is explained in~\cite[first part of the proof of
Proposition~8.4]{PS5} that the inclusion $\sX*\varinjlim\sY\subset
\varinjlim(\sX*\sY)$ holds for any two classes of objects $\sX$
and $\sY$ in a Grothendieck category~$\sK$.
 This takes care of the leftmost inclusion in~\eqref{three-inclusions}.

 Furthermore, any $\FP_2$ object is, by definition, finitely
presentable; and the class of all finitely presentable objects is
closed under extensions in $\sK$ by~\cite[Lemma~8.1]{PS5}.
 So we certainly have $\sT*\sT\subset\sK_\fp$.
 It is clear that the class $\sT*\sT$ is closed under finite direct
sums whenever the class $\sT$~is.
 For any class of finitely presentable objects $\sS$ closed under
finite direct sums in $\sK$, we have $\varinjlim\varinjlim\sS=
\varinjlim\sS$ by Proposition~\ref{varinjlim-varinjlim-closed}.
 This explains the rightmost inclusion in~\eqref{three-inclusions}.

 Finally, the middle inclusion in~\eqref{three-inclusions} is provable
similarly to the proof in~\cite{PS5}.
 Let $\sK$ be a Grothendieck abelian category, $\sX\subset\sK$ be
a class of objects, and $\sT\subset\sK$ be a class of objects such that
the functor $\Ext^1_\sK(T,{-})\:\sK\rarrow\Ab$ preserves direct limits
for all objects $T\in\sT$.
 We claim that the inclusion $(\varinjlim\sX)*\sT\subset
\varinjlim(\sX*\sT)$ holds.
 The argument from~\cite[second part of the proof of
Proposition~8.4]{PS5} applies.
\end{proof}

 In the following lemma, which is stated in the great generality of
localizing multiplicative subsets in arbitrary categories, we ignore
the distinction between sets and classes.
 The reader will not lose much by assuming the category $\sC$ to be
small.
 The aim of this lemma is to provide the details of an argument
sketched in~\cite[proof of Lemma~8.3]{PS5}.
 We will use it in the next
Lemma~\ref{kernel-closed-implies-FP-infinity}.

\begin{lem} \label{calculus-of-fractions-lemma}
 Let\/ $\sC$ be a category and\/ $\Sigma\subset\sC$ be a localizing
multiplicative subset/subclass of morphisms (i.~e., a subset of
morphisms containing the identity morphisms, closed under
the composition, and satisfying the Ore conditions).
 Let $X$, $Y\in\sC$ be two fixed objects. \par
\textup{(a)} Consider the slice category $(\sC\downarrow X)$ of
all morphisms into the object $X$ in the category\/ $\sC$, and let\/
$\Sigma_X\subset (\sC\downarrow X)$ denote the full subcategory on
the class of all morphisms $s\:R\rarrow X$ with $R\in\sC$ and
$s\in\Sigma$.
 Then the category\/ $\Sigma_X^\sop$ opposite to\/ $\Sigma_X$
is filtered.
 The natural map of sets\/ $\varinjlim_{(s\:R\to X)\in\Sigma_X}
\Hom_\sC(R,Y)\rarrow\Hom_{\sC[\Sigma^{-1}]}(X,Y)$ is bijective;
so the set\/ $\Hom_{\sC[\Sigma^{-1}]}(X,Y)$ of morphisms in the localized
category\/ $\sC[\Sigma^{-1}]$ is the filtered direct limit of the sets\/
$\Hom_\sC(R,Y)$ indexed over the category\/~$\Sigma_X^\sop$. \par
\textup{(b)} Let\/ $\sF\subset\sC$ be a full subcategory such that
for every morphism $s\:R\rarrow X$ in\/ $\sC$ with $s\in\Sigma$ there
exists a morphism $f\:F\rarrow R$ in\/ $\sC$ with $F\in\sF$
and $sf\in\Sigma$.
 Denote by\/ $\Sigma_{\sF,X}\subset\Sigma_X$ the full subcategory on
the class of all morphisms $t\:F\rarrow X$ with $t\in\Sigma$ and
$F\in\sF$.
 Then the full subcategory\/ $\Sigma_{\sF,X}^\sop\subset\Sigma_X^\sop$ is
filtered and cofinal in\/~$\Sigma_X^\sop$.
 Accordingly, the set\/ $\Hom_{\sC[\Sigma^{-1}]}(X,Y)$ can be computed as
the filtered direct limit\/ $\Hom_{\sC[\Sigma^{-1}]}(X,Y)\simeq
\varinjlim_{(t\:F\to X)\in\Sigma_{\sF,X}}\Hom_\sC(F,Y)$
indexed over\/~$\Sigma_{\sF,X}^\sop$.
\end{lem}

\begin{proof}
 Part~(a): the assertion that the category $\Sigma_X^\sop$ is filtered
(in the sense of~\cite[Definition~1.4]{AR}) whenever the class of
morphisms $\Sigma\subset\sC$ admits a calculus of right fractions
(in the sense of~\cite[Sections~I.2.2 and~I.2.4]{GZ})
is straightforward.
 The desired isomorphism of Hom sets is the opposite assertion
to~\cite[Proposition~I.2.4]{GZ}.
 Part~(b) is clear in view of the discussion in~\cite[Section~0.11 and
Exercise~1.o(3)]{AR}.
\end{proof}

\begin{lem} \label{kernel-closed-implies-FP-infinity}
 Let\/ $\sK$ be a Grothendieck category and\/ $\sS\subset\sK$ be
a class of finitely generated objects containing a set of generators
of\/ $\sK$ and closed under finite direct sums and kernels of
epimorphisms.
 Then all the objects from\/ $\sS$ are strongly finitely presentable
(type\/ $\FP_\infty$).
\end{lem}

\begin{proof}
 First of all, the category $\sK$ is locally finitely generated,
since it has a set of finitely generated generators by assumption.
 In this context, it is explained in~\cite[Proposition~V.3.4]{Sten}
that an object $S\in\sK$ is finitely presentable if and only if
the kernel of any epimorphism onto $S$ from any finitely generated
object $T\in\sK$ is finitely generated.
 Following the argument in~\cite{Sten}, based
on~\cite[Lemma~V.3.3]{Sten}, one can see that it suffices to let $T$
range over the quotient objects of finite direct sums of objects
from a chosen set $\sG$ of finitely generated generators of~$\sK$.
 Furthermore, the passage to the quotients holds automatically, and
so it suffices to let $T$ range over the finite direct sums of
objects from~$\sG$.
 In the situation at hand, choosing $\sG\subset\sS$, we conclude that
all the objects in $\sS$ are finitely presentable.

 The rest of the proof is similar to~\cite[proof of Lemma~8.3]{PS5}
and based on Lemma~\ref{calculus-of-fractions-lemma}.
 For any two objects $X$ and $Y$ in any abelian category $\sK$,
the abelian group $\Ext^n_\sK(X,Y)$ can be computed as the group of
morphisms $X\rarrow Y[n]$ in the derived category $\bD^-(\sK)$.
 The latter category can be constructed by inverting the class of
quasi-isomorphisms in the cochain homotopy category $\sC=\bH^-(\sK)$.
 As the class of quasi-isomorphisms $\Sigma$ is localizing
in $\bH^-(\sK)$, Lemma~\ref{calculus-of-fractions-lemma} provides
the description of the group $\Ext^n_\sK(X,Y)$ as the direct limit of
cohomology groups $H^n\Hom_\sK(R_\bu,Y)$, taken over the (large)
filtered category of exact complexes $\dotsb\rarrow R_2\rarrow
R_1\rarrow R_0\rarrow X\rarrow0$ in~$\sK$.
 Here the morphisms in the category of such arbitrary resolutions
$R_\bu\rarrow X$ are the usual morphisms of complexes acting by
the identity maps on the object $X$ and viewed up to cochain homotopy.
 The full subcategory of resolutions $R_\bu\rarrow X$ is cofinal
in~$\Sigma_X^\sop$.

 In the situation at hand, for any object $S\in\sS$, the full
subcategory of resolutions $T_\bu\rarrow S$ consisting of objects
$T_i\in\sS$ is cofinal in the category of all resolutions
$R_\bu\rarrow S$ with $R_i\in\sK$ by
Lemma~\ref{calculus-of-fractions-lemma}(b).
 Indeed, given a resolution $R_\bu\rarrow S$, one can construct
a resolution $T_\bu\rarrow S$ with $T_i\in\sS$ together with a morphism
of resolutions $T_\bu\rarrow R_\bu$, using the standard inductive
construction based on the observations that the full subcategory $\sS$
is self-generating (in the sense of
Section~\ref{generalized-cotorsion-periodicity-secn})
and closed under kernels of epimorphisms in~$\sK$.
 So one can compute the group $\Ext^n_\sK(S,Y)$ as the filtered
direct limit of $H^n\Hom_\sK(T_\bu,Y)$ taken over all the resolutions
$T_\bu\rarrow S$ with $T_i\in\sS$.
 Now, since all the objects of $\sS$ are finitely presentable,
the functor $\Hom_\sK(T_\bu,{-})$ takes direct limits in $\sK$ to
direct limits of complexes of abelian groups.
 It remains to recall that the functors of cohomology of a complex
of abelian groups preserve direct limits, and direct limits commute
with direct limits.
\end{proof}

\begin{prop} \label{varinjlim-kernel-closed}
 Let $\sK$ be a Grothendieck category and\/ $\sS\subset\sK$ be a class
of finitely generated objects containing a set of generators of\/ $\sK$
and closed under extensions and kernels of epimorphisms.
 Then the class of objects\/ $\varinjlim\sS\subset\sK$ is closed under
coproducts, direct limits, extensions, and kernels of epimorphisms.
\end{prop}

\begin{proof}
 By Lemma~\ref{kernel-closed-implies-FP-infinity}, all the objects
in $\sS$ are finitely presentable, and in fact even strongly finitely
presentable.
 As the class $\sS$ contains a set of generators for $\sK$, it follows
that the category $\sK$ is locally finitely presentable.
 So Proposition~\ref{varinjlim-varinjlim-closed} is applicable, telling
that the class $\varinjlim\sS$ is closed under coproducts and direct
limits in~$\sK$.
 Furthermore, Proposition~\ref{FP2-varinjlim-extension-closed} tells
that the class $\varinjlim\sS$ is closed under extensions.
 It remains to prove its closedness under kernels of epimorphisms
(cf.~\cite[Proposition~9.16]{CCS} for a somewhat related but
apparently different result).

 Let $C\rarrow D$ be an epimorphism in $\sK$ between two objects
$C$, $D\in\varinjlim\sS$.
 Then there exists a direct system $(S_i)_{i\in I}$, indexed
by a directed poset $I$, such that $S_i\in\sS$ for all $i\in I$
and $D=\varinjlim_{i\in I}S_i$.
 For every $i\in I$, consider the pullback diagram
\begin{equation} \label{direct-limits-morphism-pullback}
\begin{gathered}
 \xymatrix{
  C_i \ar@{->>}[r] \ar[d] & S_i \ar[d] \\
  C \ar@{->>}[r] & D
 }
\end{gathered}
\end{equation}
 Here $C_i$ is the pullback of the given epimorphism $C\rarrow D$
and the natural morphism to the direct limit $S_i\rarrow D$.
 As the index $i\in I$ varies, the upper lines
of~\eqref{direct-limits-morphism-pullback} form a direct system
of (epi)morphisms in $\sK$, whose direct limit is the epimorphism
$C\rarrow D$ in the lower line of the diagram.

 Choose a set of generators $\sG\subset\sS$ of the category $\sK$,
and put $H=\coprod_{G\in\sG}G$.
 For every index $i\in I$, denote by $\Xi_i$ the underlying set of
the image of the abelian group map $\Hom_\sK(H,C_i)\rarrow
\Hom_\sK(H,S_i)$ induced by the morphism $C_i\rarrow S_i$.
 Then, for every pair of indices $i<j\in I$, the transition morphism
$S_i\rarrow S_j$ induces a map of sets $\Xi_i\rarrow\Xi_j$.
 So we obtain a direct system of sets $(\Xi_i)_{i\in I}$.

 For any object $K\in\sK$ and any set $\Xi$, let us denote by
$K^{(\Xi)}$ the coproduct of $\Xi$ copies of $K$ in~$\sK$.
 Notice that the assignment $(K,\Xi)\longmapsto K^{(\Xi)}$ is
a covariant functor $\sK\times\Sets\rarrow\sK$ (i.~e., a covariant
functor of both the arguments $K\in\sK$ and $\Xi\in\Sets$).

 For every index $i\in I$ we have a natural morphism
$h_i\:H^{(\Xi_i)}\rarrow S_i$ in $\sK$.
 Since $H$ is a generator of the category $\sK$, and the morphism
$C_i\rarrow S_i$ is an epimorphism, the morphism~$h_i$ is
an epimorphism in $\sK$ as well.
 As the index~$i$ varies, the morphisms~$h_i$ form a direct system
$(h_i)_{i\in I}$ in the category of morphisms in~$\sK$.

 Let us show that the kernel $L_i$ of the morphism~$h_i$ belongs to
$\varinjlim\sS$.
 The object $H^{(\Xi_i)}$ is the coproduct of copies of all the objects
$G\in\sG$, each of them taken with the multiplicity~$\Xi_i$.
 Since the object $S_i$ is finitely generated, there exists a finite
subcoproduct in this coproduct mapping epimorphically onto~$S_i$.
 So we have a direct sum decomposition $H^{(\Xi_i)}=H'_i\oplus H''_i$,
where $H'_i$ is a finite direct sum of objects from $\sG$ and
the restriction of~$h_i$ onto $H'_i$ is an epimorphism
$h'_i\:H'_i\rarrow S_i$.
 Denote by $K_i$ the kernel of~$h'_i$.
 We have constructed a pushout diagram
$$
 \xymatrix{
  K_i \ar@{>->}[r] \ar@{>->}[d] & H'_i \ar@{->>}[r] \ar@{>->}[d]
  & S_i \ar@{=}[d] \\
  L_i \ar@{>->}[r] \ar@{->>}[d] & H_i \ar@{->>}[r] \ar@{->>}[d]
  & S_i \\
  H''_i \ar@{=}[r] & H''_i
 }
$$
 Now $H'_i\in\sS$, since the class $\sS$ is closed under finite direct
sums and $\sG\subset\sS$.
 Hence $K_i\in\sS$, as the class $\sS$ is closed under kernels
of epimorphisms.
 On the other hand, $H''_i\in\varinjlim\sS$, since the class
$\varinjlim\sS$ is closed under coproducts.
 As we already know that the class $\varinjlim\sS$ is closed under
extensions, we can conclude from the short exact sequence
$0\rarrow K_i\rarrow L_i\rarrow H''_i\rarrow0$ that $L_i\in
\varinjlim\sS$.

 Passing to the direct limit of~$h_i$ over $i\in I$, we see that
the kernel of the epimorphism
$$
 \varinjlim\nolimits_{i\in I}H^{(\Xi_i)}\lrarrow
 \varinjlim\nolimits_{i\in I}S_i=D
$$
belongs to $\varinjlim\varinjlim\sS=\varinjlim\sS$.
 We already know from
Corollary~\ref{kernel-of-coproduct-onto-direct-limit-cor} that
the kernel of
the epimorphism~\eqref{coproduct-onto-direct-limit-epimorphism}
(for $H_i=H^{(\Xi_i)}$) belongs to $\varinjlim\sS$.
 Since the class $\varinjlim\sS$ is closed under extensions, it
follows that the kernel $M$ of the composition of epimorphisms
$$
 H=\coprod\nolimits_{i\in I}H^{(\Xi_i)}\lrarrow
 \varinjlim\nolimits_{i\in I}H^{(\Xi_i)}\lrarrow D
$$
belongs to $\varinjlim\sS$.

 The final observation is that the epimorphism
$H=\coprod_{i\in I}H^{(\Xi_i)}\rarrow D$ factorizes through
the epimorphism $C\rarrow D$, essentially due to the construction
of the sets $\Xi_i$ in the beginning of this proof.
 Now we consider the pullback diagram
$$
 \xymatrix{
  & N \ar@{=}[r] \ar@{>->}[d] & N \ar@{>->}[d] \\
  M \ar@{>->}[r] \ar@{=}[d] & Q \ar@{->>}[r] \ar@{->>}[d]
  & C \ar@{->>}[d] \\
  M \ar@{>->}[r] & H \ar@{->>}[r] & D
 }
$$
where $Q$ is the pullback of the pair of epimorphisms $C\rarrow D$ and
$H\rarrow D$, while $N$ is the kernel of the morphism $C\rarrow D$.
 Since the epimorphism $H\rarrow D$ factorizes through
the epimorphism $C\rarrow D$, the short exact sequence $0\rarrow N
\rarrow Q\rarrow H\rarrow0$ splits.
 We have $M\in\varinjlim\sS$ and $C\in\varinjlim\sS$, so it follows
from the short exact sequence $0\rarrow M\rarrow Q\rarrow C\rarrow0$
that $Q\in\varinjlim\sS$.
 It remains to notice that the class $\varinjlim\sS$ is closed under
direct summands (since it is closed under direct limits) in~$\sK$.
 So $N\in\varinjlim\sS$ as $N$ is a direct summand of~$Q$.
\end{proof}

 We conclude the section by presenting formal proofs of
Propositions~A(i) and~B.

\begin{proof}[Proof of Proposition~A(i) from
Section~\ref{introd-theorem-A}]
 This is precisely the assertion of
Proposition~\ref{FP2-varinjlim-extension-closed}.
\hbadness=1800
\end{proof}

\begin{proof}[Proof of Proposition~B from
Section~\ref{introd-theorem-B}]
 Applying Proposition~\ref{varinjlim-kernel-closed} to the class
$\sS\cup\sT\subset\sK$, we see that the class $\varinjlim(\sS\cup\sT)$
is closed under coproducts, direct limits, extensions, and kernels
of epimorphisms in~$\sK$.
 So $\varinjlim(\sS\cup\sT)$ is precisely the class $\sC$ as
defined in the formulation of Theorem~B.
\end{proof}

\Section{Exact Categories of Grothendieck Type}
\label{exact-grothendieck-secn}

 In this section we recall some basic concepts of the theory of
\emph{efficient exact categories} and \emph{exact categories of
Grothendieck type}, as developed by Saor\'\i n and
\v St\!'ov\'\i\v cek~\cite{SaoSt,Sto-ICRA}.
 The exposition in \v St\!'ov\'\i\v cek's paper~\cite{Sto-ICRA} is
particularly convenient as a reference source for our purposes.

 Let $\sE$ be a category.
 By a \emph{well-ordered chain} (of morphisms) in $\sE$ one means
a direct system $(f_{\beta,\gamma}\:E_\gamma\to E_\beta)_
{0\le\gamma<\beta<\alpha}$ in $\sE$ indexed by an ordinal~$\alpha$.
 A well-ordered chain $(E_\beta)_{0\le\beta<\alpha}$ is said to be
\emph{smooth} if $E_\beta=\varinjlim_{\gamma<\beta}E_\gamma$ for
all limit ordinals $0<\beta<\alpha$.
 If the direct limit $E_\alpha=\varinjlim_{\beta<\alpha}E_\beta$ exists
in $\sE$, then the natural morphism $E_0\rarrow E_\alpha$ is said to be
the \emph{composition} of the smooth chain
$(E_\beta)_{0\le\beta<\alpha}$.
 The morphism $E_0\rarrow E_\alpha$ is also called the \emph{transfinite
composition} of the morphisms $E_\beta\rarrow E_{\beta+1}$, where
$0\le\beta<\alpha$.

 Let $\sE$ be a category, $\sD$ be a class of morphisms in $\sE$, and
$\kappa$~be a regular cardinal.
 An object $X\in\sE$ is said to be \emph{$\kappa$\+small relative
to\/~$\sD$} if, for any smooth chain $(E_\beta)_{0\le\beta<\alpha}$
indexed by an ordinal $\alpha$ of cofinality~$\ge\kappa$ with
the morphisms $E_\beta\rarrow E_{\beta+1}$ belonging to $\sD$ for
all $0\le\beta<\alpha$ and the direct limit $E_\alpha=
\varinjlim_{\beta<\alpha}E_\beta$, the induced map of sets
$$
 \varinjlim\nolimits_{\beta<\alpha}\Hom_\sE(X,E_\beta)\lrarrow
 \Hom_\sE(X,E_\alpha)
$$
is a bijection.
 An object $X\in\sE$ is called \emph{small relative to\/~$\sD$} if it
is $\kappa$\+small relative to $\sD$ for some regular cardinal~$\kappa$.

 The following definition is taken from~\cite[Definition~3.4]{Sto-ICRA}.
 The definition in~\cite[Proposition~2.6]{SaoSt} is slightly more
general (the result of~\cite[Proposition~2.7(2)]{SaoSt}
or~\cite[Proposition~5.3(2)]{Sto-ICRA} provides the comparison).
 An exact category $\sE$ is called \emph{efficient} if the following
conditions hold:
\begin{enumerate}
\renewcommand{\theenumi}{Ef\arabic{enumi}}
\setcounter{enumi}{-1}
\item $\sE$ is weakly idempotent-complete, i.~e., any pair of morphisms
$p\:X\rarrow Y$ and $i\:Y\rarrow X$ in $\sE$ with $pi=\id_Y$ arises
from a direct sum decomposition $X=Y\oplus Z$;
\item all transfinite compositions of admissible monomorphisms exist
in $\sE$, and the class of admissible monomorphisms is closed under
transfinite compositions;
\item every object of $\sE$ is small relative to the class of all
admissible monomorphisms;
\item the exact category $\sE$ has a generator, i.~e., there is
an object $G\in\sE$ such that every object $E\in\sE$ is the codomain
of an admissible epimorphism $G^{(I)}\rarrow E$ from a coproduct
$G^{(I)}$ of some set $I$ of copies of the object~$G$.
\end{enumerate}

 The definition of a \emph{filtration} or \emph{transfinitely iterated
extension} was already given in
Section~\ref{generalized-fp-projective-periodicity-I-secn} in
the context of Grothendieck abelian categories.
 It is generalized to exact categories in the following way.
 An $(\alpha+\nobreak1)$\+indexed smooth chain
$(E_\beta)_{0\le\beta\le\alpha}$ is said to be
an \emph{$\alpha$\+indexed filtration} (of the object~$E_\alpha$) if
$E_0=0$ and, for every ordinal $0\le\beta<\alpha$, the morphism
$E_\beta\rarrow E_{\beta+1}$ is an admissible monomorphism in~$\sE$.

 If this is the case, the object $E_\alpha$ is said to be
\emph{filtered by} the cokernels $S_\beta$ of the admissible
monomorphisms $E_\beta\rarrow E_{\beta+1}$.
 Alternatively, the object $E_\alpha$ is called a \emph{transfinitely
iterated extension} of the objects $(S_\beta)_{0\le\beta<\alpha}$.
 Given a class of objects $\sS\subset\sE$, the class of all objects
in $\sE$ filtered by (objects isomorphic to) objects from $\sS$
is denoted by $\Fil(\sS)\subset\sE$.

 The next definition is taken from~\cite[Definition~3.11]{Sto-ICRA}.
 An \emph{exact category of Grothendieck type} is an efficient exact
category $\sE$ satisfying the additional axiom
\begin{enumerate}
\item[(GT4)] the category $\sE$ is deconstructible in itself, i.~e.,
there exists a \emph{set} of objects $\sS\subset\sE$ such that
$\sE=\Fil(\sS)$.
\end{enumerate}

 The following result is important for our purposes.

\begin{thm} \label{grothendieck-type-enough-injectives}
 Any exact category of Grothendieck type has enough injective objects.
\end{thm}

\begin{proof}
 This is~\cite[Corollary~5.9]{Sto-ICRA}.
\end{proof}

 The next lemma is an exact category version of~\cite[Lemma~1.4]{BHP}.

\begin{lem} \label{exact-category-hereditarily-generated}
 Let\/ $\sE$ be an exact category and\/ $\sT\subset\sE$ be a class
of objects.
 Put\/ $\sB=\sT^{\perp_{\ge1}}$, and assume that every object of\/
$\sE$ is an admissible subobject of an object from\/~$\sB$
(in particular, this holds if there are enough injective objects
in\/~$\sE$).
 Then \par
\textup{(a)} ${}^{\perp_1}\sB={}^{\perp_{\ge1}}\sB\subset\sE$; \par
\textup{(b)} if the class\/ $\sA={}^{\perp_1}\sB={}^{\perp_{\ge1}}\sB$
is generating in\/ $\sE$, then $(\sA,\sB)$ is a hereditary cotorsion
pair in\/~$\sE$ (as defined in
Section~\ref{generalized-cotorsion-periodicity-secn}).
\end{lem}

\begin{proof}
 The argument from~\cite[Lemma~1.4]{BHP} applies.
 All the injective objects of $\sE$ always belong to $\sB$; so if
there are enough such injective objects, then every object of $\sE$
is an admissible subobject of an object from~$\sB$.
 In part~(b), it is helpful to keep in mind that the class
${}^{\perp_1}\sB$ is closed under coproducts in $\sE$ for any class
$\sB\subset\sE$ \,\cite[Corollary 8.3]{CoFu},
\cite[Corollary A.2]{CoSt}.
 So the conditions that any object of $\sE$ is an admissible
epimorphic image of an object from $\sA$ and that it is an admissible
epimorphic image of a coproduct of objects from $\sA$ are equivalent.
\end{proof}

 The following version of the Eklof--Trlifaj theorem for efficient
exact categories was obtained in the papers~\cite{SaoSt,Sto-ICRA}.

\begin{thm} \label{eklof-trlifaj-efficient-exact}
 Let\/ $\sE$ be an efficient exact category and $(\sA,\sB)$ be
the cotorsion pair generated by a set of objects\/ $\sS\subset\sE$.
 Then \par
\textup{(a)} If the class\/ $\sA$ is generating in\/ $\sE$,
then the cotorsion pair $(\sA,\sB)$ is complete. \par
\textup{(b)} If the class\/ $\Fil(\sS)$ is generating in\/ $\sE$,
then\/ $\sA=\Fil(\sS)^\oplus$.
\end{thm}

\begin{proof}
 This is~\cite[Corollary~2.15]{SaoSt} or~\cite[Theorem~5.16]{Sto-ICRA}.
\end{proof}

 The next proposition is the efficient exact category version
of~\cite[Proposition~1.5]{BHP}.

\begin{prop} \label{efficient-exact-hereditarily-generated}
 Let\/ $\sE$ be an efficient exact category and\/ $\sT\subset\sE$ be
a set of objects.
 Put\/ $\sB=\sT^{\perp_{\ge1}}$, and assume that the class\/ $\sB$
is cogenerating in\/~$\sE$ (in particular, this holds if\/ $\sE$ is
an exact category of Grothendieck type).
 Put\/ $\sA={}^{\perp_1}\sB= {}^{\perp_{\ge1}}\sB$, as per
Lemma~\ref{exact-category-hereditarily-generated}, and assume
that the class\/ $\sA$ is generating in\/~$\sE$.
 Then $(\sA,\sB)$ is a hereditary complete cotorsion pair in\/ $\sE$
generated by a certain set of objects\/~$\sS$.
\end{prop}

\begin{proof}
 Recall that if $\sE$ is of Grothendieck type, then there are enough
injective objects in $\sE$ by
Theorem~\ref{grothendieck-type-enough-injectives};
so $\sB$ is cogenerating in~$\sE$.
 In view of Lemma~\ref{exact-category-hereditarily-generated} and
Theorem~\ref{eklof-trlifaj-efficient-exact}, we only need to construct
a set of objects $\sS\subset\sE$ such that $\sS^{\perp_1}=
\sT^{\perp_{\ge1}}$.
 Clearly, we have $\sT\subset\sA$.
 Arguing by induction similarly to~\cite[proof of
Proposition~1.5]{BHP}, it suffices to show that for every object
$S\in\sA$ and an integer $n\ge2$ there exists a set of objects
$\sS'\subset\sA$ such that for any given $X\in\sE$ one has
$\Ext^n_\sE(S,X)=0$ whenever $\Ext^{n-1}_\sE(S',X)=0$ for
all $S'\in\sS'$.

 Let $G\in\sE$ be a generator of the exact category $\sE$, as in
condition~(Ef3).
 The result of~\cite[Proposition~5.3]{Sto-ICRA} provides, for any
given object $S\in\sE$, a set $\sJ_S$ of admissible monomorphisms
in $\sE$ satisfying the following two conditions:
\begin{enumerate}
\item every morphism $j\in\sJ_S$ fits into a short exact sequence
$0\rarrow E\overset j\rarrow G^{(I)}\rarrow S\rarrow0$ in $\sE$ for
some set~$I$;
\item for every short exact sequence $0\rarrow Z\rarrow Y\rarrow S
\rarrow0$ in $\sE$ there exists a morphism of short exact sequences
$(0\to E\to G^{(I)}\to S\to0)\rarrow(0\to Z\to Y\to S\to0)$ in $\sE$
such that the admissible monomorphism $E\rarrow G^{(I)}$ belongs to
$\sJ_S$, while $S\rarrow S$ is the identity morphism.
\end{enumerate}

 Let $\sJ_S$ be a set of admissible monomorphisms satisfying
conditions~(1\+-2) for the object $S\in\sE$.
 For every short exact sequence $0\rarrow E\overset j\rarrow H
\rarrow S\rarrow0$ in $\sE$ with $j\in\sJ_S$, choose an admissible
epimorphism $A\rarrow H$ onto $H$ from an object $A\in\sA$ (cf.\
the proof of Lemma~\ref{exact-category-hereditarily-generated}), and
set $S'$ to be the kernel of the composition $A\rarrow H\rarrow S$.
 Then one has $S'\in\sA$, since $\sA$ is closed under the kernels of
admissible epimorphisms (as the cotorsion pair $(\sA,\sB$) in $\sE$
is hereditary).

 Let $\sS'$ be the set of all objects $S'$ obtained in this way.
 For any Yoneda extension class $\xi\in\Ext^n_\sE(S,X)$, there
exists a short exact sequence $0\rarrow Z\rarrow Y\rarrow S\rarrow0$
in $\sE$ such that the class~$\xi$ is the composition of the class in
$\Ext^1$ represented by this short exact sequence with some Yoneda
extension class $\eta\in\Ext_\sE^{n-1}(Z,X)$.
 By construction, any short exact sequence $0\rarrow Z\rarrow Y
\rarrow S\rarrow0$ in $\sE$ is a pushout of a short exact sequence
$0\rarrow E\overset j\rarrow H\rarrow S\rarrow0$ with $j\in\sJ_S$,
which in turn is a pushout of a short exact sequence
$0\rarrow S'\rarrow A\rarrow S\rarrow0$ with $A\in\sA$ and $S'\in\sS'$.
 It follows easily that $\Ext_\sE^{n-1}(S',X)=0$ for all $S'\in\sS'$
implies $\Ext_\sE^n(S,X)=0$.
\end{proof}

 The next theorem of \v St\!'ov\'\i\v cek plays a key role in
our proof of Theorem~A.

\begin{thm} \label{deconstructible-class-is-grothendieck-type}
 Let\/ $\sK$ be a Grothendieck abelian category and\/ $\sE\subset\sK$
be a deconstructible class of objects (as defined in
Section~\ref{generalized-fp-projective-periodicity-I-secn}).
 Assume additionally that the full subcategory\/ $\sE$ is closed under
direct summands in\/~$\sK$.
 Then the category\/ $\sE$, endowed with the exact category structure
inherited from the abelian exact structure of\/ $\sK$, is an exact
category of Grothendieck type.
\end{thm}

\begin{proof}
 This is~\cite[Theorem~3.16]{Sto-ICRA}.
\end{proof}

\begin{lem} \label{deconstructibility-transitive}
 Let\/ $\sK$ be a efficient exact category and\/ $\sE\subset\sK$
be a full subcategory closed under transfinitely iterated extensions,
endowed with the inherited exact structure.
 Let\/ $\sS\subset\sE$ be a class of objects.
 Then the notation\/ $\Fil(\sS)$ is unambiguous: the class of all\/
$\sS$\+filtered objects in\/ $\sE$ coincides with the class of all\/
$\sS$\+filtered objects in\/~$\sK$.
\end{lem}

\begin{proof}
 This is a part of~\cite[Lemma~1.11]{SP} or~\cite[Lemma~3.18]{Sto-ICRA}.
\end{proof}

\Section{Pure Exact Structure and Deconstructibility}
\label{pure-exact-structure-secn}

 The aim of this section is to prove Proposition~A(ii).
 We restate it now in a more general form as the following
Proposition~\ref{weakly-deconstructible-prop}.

 Let $\sK$ be a Grothendieck category.
 We will say that a class of objects\/ $\sF\subset\sK$ is
\emph{weakly deconstructible} if there exists a set of objects
$\sT\subset\sF$ such that $\sF\subset\Fil(\sT)$.
 Clearly, a class of objects is deconstructible if and only if it is
weakly deconstructible \emph{and} closed under transfinitely iterated
extensions.

\begin{prop} \label{weakly-deconstructible-prop}
 Let\/ $\sK$ be a locally finitely presentable abelian category and\/
$\sS\subset\sK$ be a class of finitely presentable objects closed under
finite direct sums.
 Then the class of objects\/ $\varinjlim\sS\subset\sK$ is weakly
deconstructible.
\end{prop}

 The definition of the pure exact structure on the module category
$\Modr R$ was already given in
Section~\ref{generalized-cotorsion-periodicity-secn}.
 It is extended to finitely accessible additive categories $\sK$
as follows~\cite[Section~3]{CB}, \cite[Section~4]{Sto}.

 A \emph{pure short exact sequence} $0\rarrow K\rarrow L\rarrow M
\rarrow0$ in $\sK$ is a pair of composable morphisms such that
the functor $\Hom_\sK(T,{-})\:\sK\rarrow \Ab$ takes this sequence to
a short exact sequence of abelian groups for every finitely presentable
object $T\in\sK$.
 It is not immediately obvious from this definition that the collection
of all pure short exact sequences defines an exact structure on~$\sK$.
 This is the result of part~(a) of the next
Proposition~\ref{pure-exact-structure-proposition}.

 For any small preadditive category $\cS$ (i.~e., a small category
enriched in abelian groups), we denote by $\Modr\cS=
\Funct_\ad(\cS^\sop,\Ab)$ the category of contravariant additive
functors from $\cS$ to $\Ab$, and by $\cS\Modl=\Funct_\ad(\cS,\Ab)$
the category of covariant additive functors $\cS\rarrow\Ab$.
 A small preadditive category $\cS$ can be viewed as a ``ring with
many objects'' or ``a nonunital ring with enough idempotents'';
then the objects of $\Modr\cS$ and $\cS\Modl$ are interpreted as right
and left $\cS$\+modules.

 The abelian category $\Modr\cS$ is locally finitely presentable and
has enough projective objects.
 Representable functors play the role of free modules with one
generator in $\Modr\cS$, and the projective objects are the direct
summands of coproducts of representables.
 There is a naturally defined tensor product functor $\ot_\cS\:
\Modr\cS\times\cS\Modl\rarrow\Ab$, and its derived functor $\Tor^\cS_*$
can be constructed as usual.
 Hence one can speak of \emph{flat} right and left $\cS$\+modules.
 We denote the full subcategory of flat modules by
$\Modrfl\cS\subset\Modr\cS$ (cf.\ the discussion
in~\cite[Section~2]{BHP}).

\begin{prop} \label{pure-exact-structure-proposition}
 Let\/ $\sK$ be a finitely accessible additive category.
 In this context: \par
\textup{(a)} In any pure short exact sequence\/ $0\rarrow K\rarrow L
\rarrow M\rarrow0$ in\/ $\sK$, the morphism $K\rarrow L$ is a kernel
of the morphism $L\rarrow M$, and the morphism $L\rarrow M$ is
a cokernel of the morphism $K\rarrow L$.
 The class of all pure short exact sequences defines an exact
structure on\/ $\sK$, called the \emph{pure exact structure}. \par
\textup{(b)} The pure short exact sequences in\/ $\sK$ are precisely
all the direct limits of split short exact sequences in\/~$\sK$. \par
\textup{(c)} Let\/ $\sS\subset\sK$ be a class of finitely presentable
objects closed under finite direct sums such that all the objects of\/
$\sK$ are direct limits of objects from\/~$\sS$.
 Denote by $\cS$ a small category equivalent to the full subcategory\/
$\sS\subset\sK$.
 Then there is a natural equivalence between the category\/ $\sK$ and
the category of flat right $\cS$\+modules, $\sK\simeq\Modrfl\cS$.
 Under this equivalence, the pure exact sequences in\/ $\sK$ correspond
precisely to the short sequences in\/ $\Modrfl\cS$ that are exact
in\/ $\Modr\cS$.
 So the pure exact structure on\/ $\sK$ corresponds to the exact
structure on\/ $\Modrfl\cS$ inherited from the abelian exact structure
on\/ $\Modr\cS$.
\end{prop}

\begin{proof}
 Part~(c): the first assertion is~\cite[Theorems~1.4(2)]{CB},
\cite[Proposition~5.1]{Kra}, or~\cite[Theorem~1.1]{DG}
(cf.~\cite[Proposition~4.2]{Sto} and~\cite[Lemma~2.2]{BHP}).
 Notice that the finitely presentable objects of $\sK$ are precisely
all the direct summands of the objects from $\sS$, so $\sS$ is
a strong generating family in $\sK$ in the sense of~\cite{DG}.

 The functor $\sK\rarrow\Modrfl\cS$ assigns to an object $K\in\sK$
the contravariant functor $\Hom_\sK({-},K)\:\sK^\sop\rarrow\Ab$
restricted to the full subcategory $\sS\subset\sK$.
 This functor identifies the full subcategory $\sS\subset\sK$ with
the full subcategory of representable functors in $\Modr\cS$, and
preserves direct limits (as the objects from $\sS$ are finitely
presentable).
 For an arbitrary preadditive category $\cS$, the representable functors
play the role of free modules with one generator in $\Modr\cS$;
when $\cS$ is an additive category, as in the situation at hand,
these are the same things as the finitely generated free modules.
 It remains to recall that the flat modules are the direct limits of
finitely generated free ones (also over a ring with
many objects~\cite[Theorem~3.2]{OR}).

 To prove the second assertion of~(c), it suffices to say that
the equivalence $\sK\simeq\Modrfl\cS$, viewed as a functor $\sK\rarrow
\Modr\cS$, takes pure exact sequences in $\sK$ to exact sequences
in $\Modr\cS$ by construction.
 On the other hand, the inverse functor $\Modrfl\cS\rarrow\sK$ takes
short exact sequences of flat modules to pure short exact sequences
in $\sK$ because every short exact sequence of flat modules is a direct
limit of split short exact sequences.
 Notice that the direct limits of pure short exact sequences are pure
exact in $\sK$, as one can easily see from the definition.

 Parts~(a) and~(b) follow from part~(c), as the class of all short
sequences in $\Modrfl\cS$ that are exact in $\Modr\cS$ clearly has
all the desired properties.
 The assertion of part~(b) is also a part
of~\cite[Theorem~16.1.15]{Pre}, while part~(a) is explained
in~\cite[Section~3.1]{CB}.
\end{proof}

 With the pure exact structure in mind, one can speak about \emph{pure
subobjects}, \emph{pure quotients}, \emph{pure monomorphisms},
\emph{pure epimorphisms}, \emph{pure-projective objects},
and \emph{pure acyclic complexes} in a finitely accessible
additive category~$\sK$.

 Part~(b) of the following proposition is a generalization of
Lemma~\ref{restricted-pure-exact-structure}.

\begin{prop} \label{pure-exact-structure-subcategory}
 Let\/ $\sK$ be a finitely accessible additive category and\/
$\sS\subset\sK$ be a class of finitely presentable objects closed
under finite direct sums.
 Then \par
\textup{(a)} The full subcategory\/ $\sC=\varinjlim\sS\subset\sK$ is
closed under pure extensions (as well as pure quotients) in\/ $\sK$,
so it inherits an exact category structure from the pure exact
structure on\/~$\sK$.
 If the category\/ $\sK$ is abelian and locally finitely presentable,
then the full subcategory\/ $\sC$ is also closed under pure subobjects
in\/~$\sK$. \par
\textup{(b)} The inherited exact category structure on\/ $\sC\subset\sK$
coincides with the pure exact structure on the finitely
accessible additive category\/~$\sC$.
 So, in this exact structure, a short sequence\/ $0\rarrow K\rarrow L
\rarrow M\rarrow0$ is exact if and only if the short sequence of abelian
groups\/ $0\rarrow \Hom_\sK(S,K)\rarrow\Hom_\sK(S,L)\rarrow\Hom_\sK(S,M)
\rarrow0$ is exact for every object $S\in\sS$. \par
\textup{(c)} Denote by $\cS$ a small category equivalent to
the full subcategory\/ $\sS\subset\sK$.
 Then there is a natural equivalence between the category\/ $\sC$ and
the category of flat right $\cS$\+modules, $\sC\simeq\Modrfl\cS$.
 Under this equivalence, the exact structure on\/ $\sC$ inherited from
the pure exact structure on\/ $\sK$ corresponds to the exact structure
on\/ $\Modrfl\cS$ inherited from the abelian exact structure on\/
$\Modr\cS$.
\end{prop}

\begin{proof}
 Part~(a) is a straightforward generalization
of~\cite[Proposition~2.2]{Len} which was already mentioned in the proof
of Corollary~\ref{kernel-of-coproduct-onto-direct-limit-cor}.
 Notice that any finitely accessible additive category has arbitrary
coproducts, which makes the argument from~\cite{Len} applicable.
 For another argument, see Proposition~\ref{kappa-lenzing-prop} below.

 Part~(b): the category $\sC$ is finitely accessible
by~\cite[Theorem~4.1]{CB} or~\cite[Proposition~5.11]{Kra}.
 Furthermore, the finitely presentable objects of $\sC$ are precisely
all the direct summands of the objects from~$\sS$; so all of them
are also finitely presentable in~$\sK$.
 Now it follows immediately from the definitions that any pure short
exact sequence in $\sK$ with the terms belonging to $\sC$ is
a pure short exact sequence in~$\sC$.
 The argument from the proof of
Lemma~\ref{restricted-pure-exact-structure} proves the converse
implication.
 Alternatively, by
Proposition~\ref{pure-exact-structure-proposition}(b), any pure short
exact sequence in $\sC$ is a direct limit of split short exact
sequences in $\sC$, hence it is also a direct limit of split short
exact sequences in $\sK$, i.~e., a pure short exact sequence in~$\sK$.

 Part~(c) is deduced from part~(b) by applying
Proposition~\ref{pure-exact-structure-proposition}(c) to the class
of objects $\sS$ in the finitely accessible additive category~$\sC$.
\end{proof}

\begin{proof}[Proof of
Proposition~\ref{weakly-deconstructible-prop}]
 We have to show that there is a set of objects $\sT\subset\sC$ such
that all the objects of $\sC$ are filtered by $\sT$ in~$\sK$.
 Let us prove a stronger assertion instead: there exists a set of
objects $\sT\subset\sC$ such that all the objects of $\sC$ are
\emph{pure filtered} by $\sT$, that is, filtered by objects from $\sT$
in the pure exact structure on~$\sC$.
 Clearly, any filtration in the pure exact structure on $\sC$ is
also a filtration in the pure exact structure on $\sK$, and
consequently it is a filtration in the abelian exact structure on $\sK$
(as all pure short exact sequences in $\sK$ are exact).

 Now we use the equivalence of exact categories $\sC\simeq\Modrfl\cS$
from Proposition~\ref{pure-exact-structure-subcategory}(c).
 In view of this equivalence, it remains to observe that the class
of all flat $\cS$\+modules is deconstructible (in itself viewed as
an exact category, or equivalently, in the abelian category of modules
$\Modr\cS$, cf.\ Lemma~\ref{deconstructibility-transitive}).
 This is essentially the result of~\cite[Lemma~1]{BBE} (see
also~\cite[Lemma~6.23]{GT}).
\end{proof}

\begin{rem}
 The assertion of Proposition~\ref{weakly-deconstructible-prop} admits
a far-reaching generalization: all the finite presentability
conditions can be dropped.
 For any Grothendieck category $\sK$ and any \emph{set} of
objects $\sS\subset\sK$ closed under finite direct sums, the class
$\varinjlim\sS\subset\sK$ is weakly deconstructible.
 This is a Grothendieck category multiobject generalization
of~\cite[Corollary~3.4]{PPT}, provable by an argument similar to
the one in~\cite{PPT} and extending the proofs of
Propositions~\ref{weakly-deconstructible-prop}\+-%
\ref{pure-exact-structure-subcategory} in the following way.

 For consistency of notation, let us denote again by $\cS$ the full
additive subcategory in $\sK$ corresponding to the class~$\sS$.
 Then there is no longer a category equivalence as in
Proposition~\ref{pure-exact-structure-subcategory}(c), but there is
still a right exact, direct limit-preserving functor
$\widetilde\Theta\:\Modr\cS\rarrow\sK$ left adjoint to the restricted
Yoneda functor $K\longmapsto\Hom_\sK({-},K)|_\sS$.
 In the spirit of the argument in~\cite{PPT}, one can interpret
$\widetilde\Theta$ as a \emph{tensor product} functor.
 Specifically, this is a restriction of the category-theoretic
tensor product operation
$$
 \ot_\cS\:\Funct_\ad(\cS^\sop,\Ab)\times
 \Funct_\ad(\cS,\sK)\rarrow\sK
$$
(see~\cite[Section~1]{OR} for the definition).
 The functor $\widetilde\Theta$ is constructed by tensoring the usual
right $\cS$\+modules with one specific left $\cS$\+module given by
the covariant identity inclusion functor $\Id\:\cS=\sS\rarrow\sK$; so
$$
 \widetilde\Theta(\cM)=\cM\ot_\cS\Id
$$
for all $\cM\in\Modr\cS=\Funct_\ad(\cS^\sop,\Ab)$.
 Denoting by $\Theta\:\Modrfl\cS\rarrow\sK$ the restriction of
$\widetilde\Theta$ to the full subcategory of flat modules
$\Modrfl\cS\subset\Modr\cS$, one observes that $\Theta$ is
an exact functor (since all the short exact sequences of flat modules
are direct limits of split ones).
 The functor $\Theta$ is \emph{not} fully faithful, but the full
subcategory $\sC\subset\sK$ is the essential image of~$\Theta$
(essentially for the reasons explained in~\cite{PPT}).

 Denoting by $\cT\subset\Modrfl\cS$ a set of flat modules such that
$\Modrfl\cS=\Fil(\cT)$, one concludes that $\sT=\Theta(\cT)\subset\sK$
is a set of objects such that $\sT\subset\sC$ and
$\sC\subset\Fil(\sT)$, since the functor $\Theta$ preserves
transfinitely iterated extensions.
\end{rem}

\begin{proof}[Proof of Proposition~A(ii) from
Section~\ref{introd-theorem-A}]
 By assumption, the class $\sC=\varinjlim\sS$ is closed under
extensions in~$\sK$.
 Since $\sC$ is also closed under direct limits in $\sK$ by
Proposition~\ref{varinjlim-varinjlim-closed}, it follows that $\sC$
is closed under transfinitely iterated extensions in~$\sK$.
 Since, on the other hand, the class $\sC$ is weakly deconstructible
in $\sK$ by Proposition~\ref{weakly-deconstructible-prop},
we can conclude that $\sC$ is deconstructible under our assumptions.
\end{proof}

\begin{cor} \label{angeleri-trlifaj-cotorsion-pair-in-lfp-category}
 Let\/ $\sK$ be a locally finitely presentable abelian category,
and let\/ $\sS\subset\sK$ be a class of objects of type\/ $\FP_2$
(as defined in Section~\ref{direct-limit-closures-of-FP-classes})
containing a set of generators of the abelian category\/ $\sK$ and
closed under extensions in\/~$\sK$.
 Put\/ $\sC=\varinjlim\sS\subset\sK$.
 Then there is a complete cotorsion pair $(\sC,\sD)$ in\/~$\sK$.
\end{cor}

\begin{proof}
 By Proposition~A(i\+-ii), or in other words, by
Propositions~\ref{varinjlim-varinjlim-closed},
\ref{FP2-varinjlim-extension-closed},
and~\ref{weakly-deconstructible-prop}, the class $\sC$
is deconstructible in~$\sK$; so $\sC=\Fil(\sT)$ for a set of
objects $\sT\subset\sK$.
 Furthermore, by assumption, the class $\sC$ contains a set of
generators of~$\sK$.
 Let $(\sC',\sD)$ be the cotorsion pair in $\sK$ generated by~$\sT$.
 Applying Theorem~\ref{eklof-trlifaj-theorem}, we conclude that
$(\sC',\sD)$ is a complete cotorsion pair and $\sC'=\sC$.
\end{proof}

\Section{Classes of $\kappa$-Presentables and Their $\kappa$-Direct
Limit Closures}  \label{direct-limits-of-kappa-P-classes}

 In this section we discuss generalizations of some results from
Sections~\ref{generalized-fp-projective-periodicity-I-secn},
\ref{direct-limit-closures-of-FP-classes},
and~\ref{pure-exact-structure-secn} from the countable
cardinal $\aleph_0$ to arbitrary regular cardinals~$\kappa$.
 In particular, we present a version of Proposition~A(i) for regular
cardinals~$\kappa$, stated below as
Proposition~\ref{kappaP2-varinjlim-extension-closed}, and discuss
the difficulties involved with an attempt to extend Proposition~A(ii)
to higher cardinals (see Remark~\ref{proposition-A-ii-remark}).
 The results of this section will be used in the proof of Theorem~A
in the next Section~\ref{generalized-fp-projective-periodicity-II-secn}.

 We refer to the book~\cite[Definitions~1.13 and~1.17, and
Theorem~1.20]{AR} for the definitions of a \emph{$\kappa$\+presentable
object} and a \emph{locally $\kappa$\+presentable} category
(for a regular cardinal~$\kappa$).
 For the more general notion of a \emph{$\kappa$\+accessible category},
see~\cite[Definition~2.1]{AR}.
 The functors of $\kappa$\+direct limit (i.~e., direct limits indexed
by $\kappa$\+directed posets) are exact in any locally
$\kappa$\+presentable category~\cite[Proposition~1.59]{AR}.
 Up to an isomorphism, in a locally $\kappa$\+presentable category
there is only a set of $\kappa$\+presentable
objects~\cite[Remark~1.19]{AR}.
 Any Grothendieck abelian category is locally presentable, i.~e.,
locally presentable for \emph{some} regular cardinal~$\kappa$
\,\cite[Lemma~A.1]{Sto0}.

 The following proposition is a generalization of
Proposition~\ref{varinjlim-varinjlim-closed} to higher cardinals,
and also a category-theoretic generalization
of~\cite[Proposition~5.5]{GIT}.
 We state it for nonadditive categories, as the additive category case
is no easier than the general one.

\begin{prop} \label{kappa-varinjlim-varinjlim-closed}
 Let\/ $\sK$ be a $\kappa$\+accessible category and\/ $\sS\subset\sK$
be a class of\/ $\kappa$\+presentable objects.
 Then the class\/ $\varinjlim^{(\kappa)}\sS$ of all $\kappa$\+direct
limits of objects from\/ $\sS$ in\/ $\sK$ (i.~e., direct limits
indexed by $\kappa$\+directed posets) is closed under
$\kappa$\+direct limits in\/~$\sK$.
 An object $L\in\sK$ belongs to\/ $\varinjlim^{(\kappa)}\sS$ if and
only if, for any $\kappa$\+presentable object $T$ in\/ $\sK$, any
morphism $T\rarrow L$ in\/ $\sK$ factorizes through an object
from\/~$\sS$.
 The full subcategory\/ $\varinjlim^{(\kappa)}\sS\subset\sK$ is
$\kappa$\+accessible, and its $\kappa$\+presentable objects are
precisely all the retracts of the objects from\/~$\sS$.
 If all coproducts exist in\/ $\sK$ and the class\/ $\sS$ is closed
under $\kappa$\+small coproducts (i.~e., coproducts indexed by sets
of cardinality~$<\kappa$), then the class\/ $\varinjlim^{(\kappa)}\sS$
is closed under all coproducts in\/~$\sK$.
\end{prop}

\begin{proof}
 Two assertions need to be explained: the ``if'' implication and
the closedness under coproducts.
 The ``only if'' implication is obvious; and the closedness under
$\kappa$\+direct limits follows from the ``if and only if''.

 Concerning the ``if'' implication, the argument is similar to
the one in~\cite[Propositions~1.22 and~2.8(i\+-ii)]{AR}.
 It is convenient to use~\cite[Theorem~1.5 and Remark~1.21]{AR} to
the effect that it suffices to construct a $\kappa$\+filtered category
$D$ and a $D$\+indexed diagram $(S_d)_{d\in D}$ of objects $S_d\in\sS$
in $\sK$ such that $L=\varinjlim_{d\in D}S_d$.
 For this purpose, let $D$ be the essentially small category of all
pairs $d=(S_d,f_d)$, where $S_d\in\sS$ and $f_d\:S_d\rarrow L$ is
an arbitrary morphism.
 Morphisms in the category $D$ are defined in the obvious way, and
the construction of the diagram $D\rarrow\sK$ is also obvious.
 See also~\cite[Proposition~1.2]{Pacc}.

 Concerning the coproducts, let $(I_\xi)_{\xi\in\Xi}$ be a family of
$\kappa$\+directed posets, indexed by a set~$\Xi$; and let
$(K_{i,\xi})_{i\in I}$ be a diagram in $\sK$, indexed by the poset
$I_\xi$ and given for every $\xi\in\Xi$.
 Then the coproduct of $\kappa$\+direct limits
$\coprod_{\xi\in\Xi}\varinjlim_{i\in I_\xi}K_{i,\xi}$ can be expressed
as the following $\kappa$\+direct limit of $\kappa$\+small coproducts.
 Denote by $J$ the set of all pairs $j=(\Upsilon,t)=(\Upsilon_j,t_j)$, 
where $\Upsilon\subset\Xi$ is a subset of cardinality smaller
than~$\kappa$ and $t\:\Upsilon\rarrow\coprod_{\upsilon\in\Upsilon}
I_\upsilon$ is a function assigning to every element $\upsilon\in
\Upsilon$ an element $t(\upsilon)\in I_\upsilon$.
 Given two elements $j$ and $k\in J$, we say that $j\le k$ if
$\Upsilon_j\subset\Upsilon_k$ and, for every $\upsilon\in\Upsilon_j$,
the inequality $t_j(\upsilon)\le t_k(\upsilon)$ holds in~$I_\upsilon$.
 Then $J$ is a $\kappa$\+directed poset; and it is easy to construct
a natural $J$\+indexed diagram in $\sK$, with the object
$K_j=\coprod_{\upsilon\in\Upsilon_j}K_{t_j(\upsilon),\upsilon}$ sitting
at the vertex $j\in J$, such that $\varinjlim_{j\in J}K_j=
\coprod_{\xi\in\Xi}\varinjlim_{i\in I_\xi}K_{i,\xi}$.
\end{proof}

 Let\/ $\sK$ be a $\kappa$\+accessible additive category.
 A \emph{$\kappa$\+pure short exact sequence} $0\rarrow K\rarrow L
\rarrow M\rarrow0$ in $\sK$ is a pair of composable morphisms such that
the functor $\Hom_\sK(T,{-})\:\sK\rarrow\Ab$ takes this sequence to
a short exact sequence of abelian groups for every $\kappa$\+presentable
object $T\in\sK$.
 Any $\kappa$\+direct limit of $\kappa$\+pure short exact sequences
is a $\kappa$\+pure short exact sequence.
 It is not immediately obvious that the collection of all
$\kappa$\+pure short exact sequences defines an exact structure
on~$\sK$; this is the result of part~(a) of the next
Proposition~\ref{kappa-pure-exact-structure-proposition}.
 We refer to the papers~\cite{AR2} and~\cite[Section~4]{Plce}
for some details.

 Let $R$ be an associative ring.
 One important particular case of the construction of the class
$\varinjlim^{(\kappa)}\sS$ occurs when $\sK=\Modr R$ is the module
category and $\sS$ is the class of all free (or projective)
$R$\+modules with less than~$\kappa$ generators.
 An $R$\+module is said to be \emph{$\kappa$\+flat} if it can
be presented as a $\kappa$\+direct limit of projective $R$\+modules,
or equivalently, as a $\kappa$\+direct limit of free $R$\+modules
with less than~$\kappa$ generators~\cite[Theorem~6.1]{Pper}.
 The class of all $\kappa$\+flat $R$\+modules is closed under
extensions and the kernels of epimorphisms in $\Modr R$; and any
short exact sequence of $\kappa$\+flat $R$\+modules is
a $\kappa$\+direct limit of split short exact sequences of
$\kappa$\+flat $R$\+modules~\cite[Lemma~6.2]{Pper}.

 More generally, let $\cS$ be a small preadditive category.
 Then a right $\cS$\+module (i.~e., an additive functor
$\cS^\sop\rarrow\Ab$) is said to be \emph{$\kappa$\+flat} if it
can be presented as a $\kappa$\+direct limit of projective
$\cS$\+modules, or equivalently, as a $\kappa$\+direct limit of
coproducts of less than~$\kappa$ representable functors.
 We will denote the full subcategory of $\kappa$\+flat modules by
$\Modrkapfl\cS\subset\Modr\cS$.
 The same results from~\cite[Section~6]{Pper} remain valid in this
context; so a short sequence in $\Modrkapfl\cS$ is exact in
the abelian category $\Modr\cS$ if and only if it is $\kappa$\+pure
exact in $\Modr\cS$, and if and only if it is $\kappa$\+pure exact
in the $\kappa$\+accessible additive category $\Modrkapfl\cS$.

\begin{prop} \label{kappa-pure-exact-structure-proposition}
 Let\/ $\sK$ be a $\kappa$\+accessible additive category.
 In this context: \par
\textup{(a)} In any $\kappa$\+pure short exact sequence\/ $0\rarrow K
\rarrow L\rarrow M\rarrow0$ in\/ $\sK$, the morphism $K\rarrow L$ is
a kernel of the morphism $L\rarrow M$, and the morphism $L\rarrow M$ is
a cokernel of the morphism $K\rarrow L$.
 The class of all $\kappa$\+pure short exact sequences defines an exact
structure on\/~$\sK$. \par
\textup{(b)} The $\kappa$\+pure short exact sequences in\/ $\sK$ are
precisely all the $\kappa$\+direct limits of split short exact sequences
in\/~$\sK$. \par
\textup{(c)} Assume that all coproducts exist in\/ $\sK$, and let\/
$\sS\subset\sK$ be a class of $\kappa$\+presentable objects closed
under $\kappa$\+small coproducts such that all the objects of\/ $\sK$
are $\kappa$\+direct limits of objects from\/~$\sS$.
 Denote by $\cS$ a small category equivalent to the full subcategory\/
$\sS\subset\sK$.
 Then there is a natural equivalence between the category\/ $\sK$ and
the category of $\kappa$\+flat right $\cS$\+modules,
$\sK\simeq\Modrkapfl\cS$.
 Under this equivalence, the $\kappa$\+pure exact sequences in\/ $\sK$
correspond precisely to the short sequences in\/ $\Modrkapfl\cS$ that
are exact in\/ $\Modr\cS$.
 So the $\kappa$\+pure exact structure on\/ $\sK$ corresponds to
the exact structure on\/ $\Modrkapfl\cS$ discussed above.
\end{prop}

\begin{proof}
 This is a $\kappa$\+version of
Proposition~\ref{pure-exact-structure-proposition}.
 Parts~(a\+-b): the argument is based on the results
of~\cite[Section~4]{Plce}.
 In the terminology of~\cite{Plce}, the $\kappa$\+pure exact structure
on a $\kappa$\+accessible additive category $\sK$ is given by the class
of all $\kappa$\+direct limits of split short exact sequences in $\sK$,
or equivalently, all $\kappa$\+direct limits of split short exact
sequences of $\kappa$\+presentable objects in~$\sK$.
 According to~\cite[Proposition~4.2]{Plce}, a morphism $L\rarrow M$
in $\sK$ is an admissible epimorphism in this exact structure if and
only if, for every $\kappa$\+presentable object $S$ in $\sK$, every
morphism $S\rarrow M$ can be lifted to a morphism $S\rarrow L$.

 Let $0\rarrow K\rarrow L\rarrow M\rarrow0$ be a short exact sequence
in the exact structure of~\cite[Section~4]{Plce}.
 Then the morphism $K\rarrow L$ is a kernel of the morphism $L\rarrow
M$; hence, in particular, for any $\kappa$\+presentable object
$S\in\sK$, the sequence of abelian groups $0\rarrow\Hom_\sK(S,K)\rarrow
\Hom_\sK(S,L)\rarrow\Hom_\sK(S,M)$ is left exact.
 On the other hand, by~~\cite[Proposition~4.2]{Plce}, every morphism
$S\rarrow M$ can be lifted to a morphism $S\rarrow L$; in other words,
the map $\Hom_\sK(S,L)\rarrow\Hom_\sK(S,M)$ is surjective.
 Thus $0\rarrow\Hom_\sK(S,K)\rarrow\Hom_\sK(S,L) \rarrow\Hom_\sK(S,M)
\rarrow0$ is a short exact sequence.
 So $0\rarrow K\rarrow L\rarrow M\rarrow0$ is a $\kappa$\+pure short
exact sequence in the sense of the definition above.

 Let us prove the converse implication.
 Let $0\rarrow K\rarrow L\rarrow M\rarrow0$ be a $\kappa$\+pure short
exact sequence in the sense of the definition above.
 This means that $0\rarrow\Hom_\sK(S,K)\rarrow\Hom_\sK(S,L)\rarrow
\Hom_\sK(S,M)\rarrow0$ is a short exact sequence for every
$\kappa$\+presentable object $S\in\sK$.
 So, in particular, the map $\Hom_\sK(S,L)\rarrow\Hom_\sK(S,M)$ is
surjective; in other words, every morphism $S\rarrow M$ can be lifted
to a morphism $S\rarrow M$.
 By~\cite[Proposition~4.2]{Plce}, the morphism $L\rarrow M$ is
an admissible epimorphism in the exact structure
of~\cite[Section~4]{Plce}.

 In order to show that the whole sequence $0\rarrow K\rarrow L\rarrow M
\rarrow0$ is short exact in the exact structure
of~\cite[Section~4]{Plce}, it remains to check that
the morphism $K\rarrow L$ is a kernel of the morphism $L\rarrow M$.
 Indeed, it follows from~\cite[Proposition~4.2]{Plce} that the morphism
$L\rarrow M$ (being an admissible epimorphism in an exact category
structure) has a kernel $K'\rarrow L$ in~$\sK$.
 It is clear that the composition $K\rarrow L\rarrow M$ vanishes
(since $K$ is a $\kappa$\+direct limit of $\kappa$\+presentable
objects and the composition $S\rarrow K\rarrow L\rarrow M$ vanishes
for every $\kappa$\+presentable object $S$ and any morphism
$S\rarrow K$ in~$\sK$).
 Hence the morphism $K\rarrow L$ factorizes naturally as $K\rarrow K'
\rarrow L$.
 For any $\kappa$\+presentable object $S\in\sK$, both the short
sequences of abelian groups $0\rarrow\Hom_\sK(S,K)\rarrow\Hom_\sK(S,L)
\rarrow\Hom_\sK(S,M)\rarrow0$ and $0\rarrow\Hom_\sK(S,K')\rarrow
\Hom_\sK(S,L)\rarrow\Hom_\sK(S,M)\rarrow0$ are exact.
 So the morphism $K\rarrow K'$ induces an isomorphism $\Hom_\sK(S,K)
\rarrow\Hom_\sK(S,K')$.
 By~\cite[Section~0.4, Definition~1.23, and Proposition~2.8(i)]{AR},
it follows that the morphism $K\rarrow K'$ is an isomorphism in~$\sK$.

 The proof of part~(c) is similar to the proof of
Proposition~\ref{pure-exact-structure-proposition}(c).
\end{proof}

 The \emph{$\kappa$\+pure exact structure} on $\sK$ is defined by
the collection of all $\kappa$\+pure short exact sequences, as per
Proposition~\ref{kappa-pure-exact-structure-proposition}(a).
 So one can speak about \emph{$\kappa$\+pure subobjects},
\emph{$\kappa$\+pure quotients}, \emph{$\kappa$\+pure acyclic
complexes}, \emph{$\kappa$\+pure-projective objects}, etc.\ (similarly
to Section~\ref{pure-exact-structure-secn}).

 The following proposition can be compared
with~\cite[Corollary~2.36]{AR}.
 Notice that, while the assumptions of the last assertion of
Proposition~\ref{kappa-lenzing-prop} are more restrictive than
those of~\cite[Corollary~2.36]{AR}, the conclusion is also stronger,
in that closedness under $\kappa$\+pure subobjects for
the cardinal~$\kappa$ appearing in the assumptions (rather than
for some possibly larger cardinal~$\lambda$) is claimed.

\begin{prop} \label{kappa-lenzing-prop}
 Let\/ $\sK$ be a $\kappa$\+accessible additive category and\/
$\sS\subset\sK$ be a class of $\kappa$\+presentable objects closed
under finite direct sums.
 Then the class of objects\/ $\sC=\varinjlim^{(\kappa)}\sS$ is closed
under $\kappa$\+pure quotients and $\kappa$\+pure extensions in\/~$\sK$.
 If the category\/ $\sK$ is abelian and locally $\kappa$\+presentable,
and the class\/ $\sS$ is closed under $\kappa$\+small coproducts in\/
$\sK$, then the full subcategory\/ $\sC$ is also closed under
$\kappa$\+pure subobjects in\/~$\sK$.
\end{prop}

\begin{proof}
 This is the $\kappa$\+version of~\cite[Proposition~2.2]{Len}
and Proposition~\ref{pure-exact-structure-subcategory}(a) above.
 Firstly, let $L\rarrow M$ be a $\kappa$\+pure epimorphism in~$\sK$.
 Assuming that $L\in\sC$, we have to show that $M\in\sC$.
 Indeed, let $T\in\sK$ be a $\kappa$\+presentable object.
 Then any morphism $T\rarrow M$ in $\sK$ can be lifted to a morphism
$T\rarrow L$, and the latter morphism factorizes through an object
from~$\sS$.
 Hence the original morphism $T\rarrow M$ also factorizes through
the same object from~$\sS$.
 By the criterion of Proposition~\ref{kappa-varinjlim-varinjlim-closed},
it follows that $M\in\sC$.

 Secondly, let $0\rarrow K\overset j\rarrow L\overset p\rarrow M
\rarrow0$ be a $\kappa$\+pure short exact sequence in $\sK$ with
$K$, $M\in\sC$.
 Let $T\in\sK$ be a $\kappa$\+presentable object and $f\:T\rarrow L$
be a morphism in~$\sK$.
 Then the composition $pf\:T\rarrow L\rarrow M$ factorizes through
an object $S\in\sS$, hence we obtain two morphisms $s\:T\rarrow S$
and $g\:S\rarrow M$ such that $pf=gs$.
 The morphism $g\:S\rarrow M$ can be lifted to a morphism
$g'\:S\rarrow L$; so we have $g=pg'$.
 Consider the difference $f-g's\:T\rarrow L$.
 Now $p(f-g's)=gs-gs=0$.
 Since the $\kappa$\+pure monomorphism~$j$ is the kernel of
the $\kappa$\+pure epimorphism~$p$ by
Proposition~\ref{kappa-pure-exact-structure-proposition}(a),
the morphism $f-g's$ factorizes through~$j$.
 Hence we have a morphism $h\:T\rarrow K$ such that $f-g's=jh$.
 The morphism~$h$ factorizes through an object $S'\in\sS$; so we obtain
two morphisms $s'\:T\rarrow S'$ and $h'\:S'\rarrow K$ such that
$h=h's'$.
 Finally, we see that the morphism $f\:T\rarrow L$ factorizes into
the composition $T\rarrow S\oplus S'\rarrow L$ of the morphism
$(s,s')\:T\rarrow S\oplus S'$ and the morphism
$(g',jh')\:S\oplus S'\rarrow L$.
 Indeed, $g's+jh's'=g's+jh=f$.
 Applying the criterion of
Proposition~\ref{kappa-varinjlim-varinjlim-closed} again,
we conclude that $L\in\sC$.

 The assertion about $\kappa$\+pure subobjects is provable by
the argument from~\cite[Proposition~2.2]{Len}.
 It is helpful to keep in mind that the class of all
$\kappa$\+presentable objects in $\sK$ is closed under colimits of
diagrams with less than~$\kappa$ vertices and
arrows~\cite[Proposition~1.16]{AR}.
\end{proof}

 The following proposition is the $\kappa$\+version of
Proposition~\ref{pure-exact-structure-subcategory}.

\begin{prop} \label{kappa-pure-exact-structure-subcategory}
 Let\/ $\sK$ be a $\kappa$\+accessible additive category and\/
$\sS\subset\sK$ be a class of $\kappa$\+presentable objects closed under
finite direct sums.
 Then \par
\textup{(a)} The full subcategory\/ $\sC=\varinjlim^{(\kappa)}\sS
\subset\sK$ inherits an exact category structure from
the $\kappa$\+pure exact structure on\/~$\sK$. \par
\textup{(b)} The inherited exact category structure on\/
$\sC\subset\sK$ coincides with the $\kappa$\+pure exact structure on
the $\kappa$\+accessible additive category\/~$\sC$.
 So, in this exact structure, a short sequence is exact if and only if
the functor\/ $\Hom_\sK(S,{-})$ takes it to a short exact sequence of
abelian groups for every object $S\in\sS$. \par
\textup{(c)} Assuming that all coproducts exist in\/ $\sK$ and
the class\/ $\sS$ is closed under $\kappa$\+small coproducts, denote by
$\cS$ a small category equivalent to the full subcategory\/
$\sS\subset\sK$.
 Then there is a natural equivalence between the category\/ $\sC$ and
the category of $\kappa$\+flat right $\cS$\+modules, $\sC\simeq
\Modrkapfl\cS$.
 Under this equivalence, the exact structure on\/ $\sC$ inherited from
the $\kappa$\+pure exact structure on\/ $\sK$ corresponds to
the exact structure on\/ $\Modrkapfl\cS$ discussed above.
\end{prop}

\begin{proof}
 Part~(a) holds by Proposition~\ref{kappa-lenzing-prop}.
 The proof of parts~(b\+-c) is similar to the proof of
Proposition~\ref{pure-exact-structure-subcategory}(b\+-c) and based on
Proposition~\ref{kappa-pure-exact-structure-proposition}.
 For the assertion that the category $\sC$ is $\kappa$\+accessible
(and the $\kappa$\+presentable objects of $\sC$ are precisely all
the direct summands of the objects from~$\sS$), use
Proposition~\ref{kappa-varinjlim-varinjlim-closed}.
\end{proof}

 The next corollary is a generalization of
Corollary~\ref{kernel-of-coproduct-onto-direct-limit-cor}.

\begin{cor} \label{kernel-of-coproduct-onto-direct-limit-kappa-version}
 Let\/ $\sK$ be a locally\/ $\kappa$\+presentable abelian category
and\/ $\sS\subset\sK$ be a class of\/ $\kappa$\+presentable objects
closed under $\kappa$\+small coproducts.
 Let $(H_i)_{i\in I}$ be a $\kappa$\+direct system of objects $H_i\in
\varinjlim^{(\kappa)}\sS$, indexed by a $\kappa$\+directed poset~$I$.
 Then the kernel of the natural epimorphism\/
$\coprod_{i\in I}H_i\rarrow\varinjlim_{i\in I}H_i$
\,\eqref{coproduct-onto-direct-limit-epimorphism}
belongs to\/ $\varinjlim^{(\kappa)}\sS$.
\end{cor}

\begin{proof}
 Both the proofs of
Corollary~\ref{kernel-of-coproduct-onto-direct-limit-cor} can be
readily adopted to the situation at hand: one can use either
Proposition~\ref{kappa-lenzing-prop}, or the construction
from~\cite[proof of Proposition~4.1]{BPS} together with
Proposition~\ref{kappa-varinjlim-varinjlim-closed}.
\end{proof}

 Let $\sK$ be an abelian category with exact functors of
$\kappa$\+direct limits, and let $n\ge1$ be an integer.
 We will say that an object $S\in\sK$ is \emph{of type
$\kappaP_n$} if the functors $\Ext_\sK^i(S,{-})\:\sK\rarrow\Ab$
preserve $\kappa$\+direct limits for $0\le i\le n-1$.
 So the objects of type $\kappaP_1$ are, by the definition,
the $\kappa$\+presentable ones.

 One can further define types $\kappaP_0$ and $\kappaP_\infty$,
similarly to the discussion in
Section~\ref{direct-limit-closures-of-FP-classes}, but we will not
need these definitions.
 The following proposition refers to objects of type $\kappaP_2$,
which form a subclass of the class of $\kappa$\+presentable objects.

\begin{prop} \label{kappaP2-varinjlim-extension-closed}
 Let\/ $\sK$ be a locally\/ $\kappa$\+presentable abelian
category and\/ $\sS$ be a class of (some) objects of type
$\kappaP_2$ closed under extensions in\/~$\sK$.
 Then the class of objects\/ $\varinjlim^{(\kappa)}\sS$ is also closed
under extensions in\/~$\sK$.
\end{prop}

\begin{proof}
 This is a straightforward generalization of
Proposition~\ref{FP2-varinjlim-extension-closed}, provable in
the similar way.
 The class of all $\kappa$\+presentable objects in $\sK$ is closed
under extensions by~\cite[Lemma~A.4]{Sto0} (the running assumption
in~\cite{Sto0} that the category is Grothendieck is not needed for
this lemma).

 The following observations play the key role.
 Let $\sK$ be an abelian category with exact functors of
$\kappa$\+direct limits.
 Then
\begin{enumerate}
\renewcommand{\theenumi}{\roman{enumi}}
\item for any two classes of objects $\sX$ and $\sY$ in $\sK$, one has
$\sX*\varinjlim^{(\kappa)}\sY\subset\varinjlim^{(\kappa)}(\sX*\sY)$;
\item for any class of objects $\sX\subset\sK$ and any class of 
objects $\sT\subset\sK$ such that the functor
$\Ext_\sK^1(T,{-})\:\sK\rarrow\Ab$ preserves $\kappa$\+direct limits
for all $T\in\sT$, one has
$(\varinjlim^{(\kappa)}\sX)*\sT\subset\varinjlim^{(\kappa)}(\sX*\sT)$.
\end{enumerate}
 Once again, the arguments from~\cite[proof of Proposition~8.4]{PS5}
apply.
\end{proof}

 The next theorem is a generalization of
Theorem~\ref{neeman-stovicek-pure-projective-pure-acyclic}
suggested in~\cite[Remark~4.11]{BHP}.

\begin{thm} \label{kappa-neeman-stovicek-theorem}
 Let\/ $\sK$ be a locally $\kappa$\+presentable abelian category.
 Let $P^\bu$ be a complex of $\kappa$\+pure-projective objects,
and let $X^\bu$ be a $\kappa$\+pure acyclic complex in\/~$\sK$.
 Then any morphism of complexes $P^\bu\rarrow X^\bu$ is homotopic
to zero.
\end{thm}

\begin{proof}
 Let $\sS$ be the class of all $\kappa$\+presentable objects in~$\sK$.
 Applying Proposition~\ref{kappa-pure-exact-structure-proposition}(c),
we conclude that the exact category $\sK$ with the $\kappa$\+pure
exact structure is equivalent to the exact category $\Modrkapfl\cS$ of
$\kappa$\+flat right $\cS$\+modules.

 Notice that all the projective $\cS$\+modules are $\kappa$\+flat, and
the kernel of any surjective morphism from a projective $\cS$\+module
to a $\kappa$\+flat one is $\kappa$\+flat by~\cite[Lemma~6.2(a)]{Pper}.
 It follows that there are enough projective objects in the exact
category $\Modrkapfl\cS$, and these projective objects are precisely
the projective $\cS$\+modules.

 Hence the equivalence of exact categories $\sK\simeq\Modrkapfl\cS$
takes the $\kappa$\+pure-projective objects of $\sK$ to the projective
$\cS$\+modules.
 It also takes $\kappa$\+pure acyclic complexes to acyclic complexes
in the exact category $\Modrkapfl\cS$, which means acyclic complexes of
$\kappa$\+flat $\cS$\+modules with $\kappa$\+flat modules of cocycles.
 As all $\kappa$\+flat modules are flat, it remains to
apply~\cite[Theorem~8.6(iii)$\Rightarrow$(i)]{Neem}
or~\cite[Theorem~4.4]{BHP}.
\end{proof}

 The following assertion extends
Proposition~\ref{stovicek-countable-hill-prop} to arbitrary regular
cardinals~$\kappa$.

\begin{prop} \label{stovicek-general-hill-prop}
 Let $\sK$ be a locally $\kappa$\+presentable Grothendieck category.
 Let\/ $\sS\subset\sK$ be a class of\/ $\kappa$\+presentable objects
closed under transfinitely iterated extensions of families of objects
of cardinality\/~$<\kappa$ (i.~e., indexed by ordinals\/
$\alpha<\kappa$).
 Let $A^\bu\in\bC(\Fil(\sS))$ be a complex in\/ $\sK$ whose terms are\/
$\sS$\+filtered objects.
 Then the complex $A^\bu$, viewed as an object of the abelian category
of complexes\/ $\bC(\sK)$, is filtered by bounded below complexes
whose terms belong to\/~$\sS$.
\end{prop}

\begin{proof}
 This is still~\cite[(proof of) Proposition~4.3]{Sto0}.
 Once again, the argument is based on
the Hill lemma~\cite[Theorem~2.1]{Sto0}.
\end{proof}

\begin{rem} \label{proposition-A-ii-remark}
 Let $\sK=\Modr R$ be the module category and $\sS\subset\Modr R$ be
a class of objects of type $\kappaP_2$ (or $\kappaP_\infty$) closed
under transfinitely iterated extensions indexed by ordinals smaller
than~$\kappa$.
 Then Propositions~\ref{kappa-varinjlim-varinjlim-closed}
and~\ref{kappaP2-varinjlim-extension-closed} tell that
the class of modules $\varinjlim^{(\kappa)}\sS$ is closed under
extensions, coproducts, and $\kappa$\+direct limits in $\Modr R$.
 But is it closed under transfinitely iterated extensions?
 For $\kappa=\aleph_0$ we said, in the proof of Proposition~A(ii) in
Section~\ref{pure-exact-structure-secn}, that transfinitely
iterated extensions are built up from extensions and direct limits.
 But this requires all direct limits (of chains of monomorphisms)
and \emph{not} only $\kappa$\+direct limits.

 On the other hand, does the analogue of
Proposition~\ref{weakly-deconstructible-prop} hold for~$\kappa$?
 In other words, is the class $\varinjlim^{(\kappa)}\sS$ weakly
deconstructible?
 Arguing similarly to the proof of
Proposition~\ref{weakly-deconstructible-prop} and using
Proposition~\ref{kappa-pure-exact-structure-subcategory}, it would be
sufficient to know that the class of $\kappa$\+flat $\cS$\+modules
is weakly deconstructible.
 But is this true?

 Furthermore, the class $\sC=\varinjlim^{(\kappa)}\sS$ is
a \emph{Kaplansky class} in the sense of~\cite{Gil0,ST0}: for
any regular cardinal~$\lambda$ there exists a regular cardinal~$\mu$
such that for any object $C\in\sC$ and any $\lambda$\+presentable
subobject $X\subset C$ there exists a $\mu$\+presentable subobject
$K\subset C$ such that $X\subset K$ and both the objects $K$ and
$C/K$ belong to~$\sC$.
 This is provable using Proposition~\ref{kappa-lenzing-prop} and
a suitable version of purification procedure (cf.~\cite[first
paragraph of the proof of Theorem~5]{BBE}, \cite[Lemma~10.5]{GT},
or~\cite[Lemma~4.1]{CS}).
 Still, the class $\sC$ is \emph{not} closed under direct limits
in general, but only under $\kappa$\+direct limits;
so~\cite[Lemma~6.9]{HT} or~\cite[Lemma~2.5(2)]{ST0} cannot be used
in order to deduce deconstructibility of~$\sC$
(cf.~\cite[Sections~10.1\+-10.2]{GT}).

 Let us point out some partial answers to the questions above that
are available in the literature.
 To begin with, we observe that the answers to the questions in
the first two paragraphs of this remark \emph{cannot} both be always
positive: the class $\varinjlim^{(\kappa)}\sS$ is \emph{not}
deconstructible in general.
 Certainly not in the context of module categories $\sK=\Modr\cT$
over small preadditive categories (or ``nonunital rings with enough
idempotents'')~$\cT$.
 Indeed, let $\sS$ be the class of projective $\cT$\+modules with
less than~$\kappa$ generators; so $\varinjlim^{(\kappa)}\sS$ is
the class of $\kappa$\+flat $\cT$\+modules.
 Suppose that the class of $\kappa$\+flat $\cT$\+modules is
deconstructible in $\Modr\cT$.
 Then, by Theorem~\ref{deconstructible-class-is-grothendieck-type},
the exact category $\Modrkapfl\cT$ would be of Grothendieck type.
 By Theorem~\ref{grothendieck-type-enough-injectives}, it would follow
that there are enough injective objects in the exact category
$\Modrkapfl\cT$.

 Take $\cT$ to be a small category equivalent to the category of
$\kappa$\+presented $R$\+modules for a given ring~$R$.
 Then, by Proposition~\ref{kappa-pure-exact-structure-proposition}(c),
the exact category $\Modr R$ with the $\kappa$\+pure exact
structure is equivalent to\/ $\Modrkapfl\cT$.
 So it would follow that there exist enough $\kappa$\+pure-injective
$R$\+modules.
 This is known to be \emph{not} true.
 See~\cite[Proposition~1.4, Remark~1.6, and Example~1.7]{ST}
(also~\cite[Theorem~6.3]{CS}).
 Therefore, the assertion of Proposition~A(ii) is \emph{cannot} be
extended straightforwardly to regular cardinals $\kappa>\aleph_0$
in general.

 Nevertheless, the class of $\kappa$\+flat $R$\+modules may be
deconstructible for \emph{some} cardinals $\kappa>\aleph_0$.
 In particular, \cite[Theorem~3.3]{ST} claims that all $\kappa$\+flat
$R$\+modules are projective if $\kappa$~is greater or equal to
a strongly compact cardinal that is greater than the cardinality of
a ring~$R$.
 So the class of $\kappa$\+flat $R$\+modules is deconstructible in
this case by Kaplansky's theorem~\cite[Corollary~7.14]{GT}.

 On the other hand, consider the case of the cardinal $\kappa=\aleph_1$.
 In this context, the class of \emph{flat Mittag-Leffler
modules}~\cite{RG,Dr,HT,ST0,Sar} plays an important role.
 Any flat Mittag-Leffler module is an $\aleph_1$\+direct limit
(in other words, an $\aleph_1$\+direct union) of its projective
submodules~\cite[Corollary~2.10]{HT}, \cite[Corollary~3.19]{GT};
so any flat Mittag-Leffler module is $\aleph_1$\+flat.
 The converse is not true in general~\cite[Example~3.5]{ST}.
 However, over a left Noetherian ring $R$, the class of flat
Mittag-Leffler right $R$\+modules is closed under $\aleph_1$\+pure
epimorphic images~\cite[Proposition~3.4]{ST}, hence under
$\aleph_1$\+direct limits; so it coincides with the class of
$\aleph_1$\+flat right $R$\+modules.

 Over any ring, the class of flat Mittag-Leffler modules is closed
under pure submodules and transfinitely iterated
extensions~\cite[Corollary~3.20(a)]{GT}, and it is a Kaplansky
class~\cite[Theorem~1.2(i) or~3.3]{ST0}, \cite[Theorem~10.6]{GT}.
 However, if a ring $R$ is not right perfect, then the class of
flat Mittag-Leffler right $R$\+modules is \emph{not}
deconstructible~\cite[Corollary~7.3]{HT}, \cite[Theorem~10.13]{GT};
it fact, it is not even precovering~\cite[Theorem~3.3]{Sar}
(cf.~\cite[Theorem~7.21]{GT}).
 So, if $R$ is not right perfect, then the class of flat Mittag-Leffler
modules is not weakly deconstructible.

 We can conclude that, for any ring $R$ that is left Noetherian but
not right perfect, the class of $\aleph_1$\+flat right $R$\+modules
is not weakly deconstructible.
\end{rem}

\Section{Generalized Flat/Projective and Fp-projective Periodicity~II}
\label{generalized-fp-projective-periodicity-II-secn}

 In this section we prove Theorem~A.
 It is restated below as
Theorem~\ref{theorem-A-detailed-in-two-parts}(a).
 The argument is a more complicated version of the proof of Theorem~0(a)
given in Section~\ref{generalized-fp-projective-periodicity-I-secn}.
 It still follows the ideas of the proof of~\cite[Theorem~0.14
or~4.1]{BHP} together with~\cite[Remark~4.11]{BHP}.

\begin{thm} \label{theorem-A-detailed-in-two-parts}
 Let\/ $\sK$ be a Grothendieck category, and let $\kappa$~be a regular
cardinal such that\/ $\sK$ is a locally $\kappa$\+presentable category.
 Let\/ $\sS\subset\sK$ be a class of (some) $\kappa$\+presentable
objects closed under transfinitely iterated extensions indexed by
ordinals smaller than~$\kappa$.
 Put\/ $\sC=\varinjlim^{(\kappa)}\sS\subset\sK$, and denote by\/
$\sA=\Fil(\sS)^\oplus$ the class of all direct summands of transfinitely
iterated extensions of objects from\/ $\sS$ in\/~$\sK$. \par
\textup{(a)} Assume that the class\/ $\sC$ is deconstructible
in\/~$\sK$.
 Put\/ $\sB'=\sS^{\perp_{\ge1}}\cap\sC$ and
$\sA'=\sC\cap{}^{\perp_1}\sB'=\sC\cap{}^{\perp_{\ge1}}\sB'$
 (so\/ $\sA\subset\sA'\subset\sC$).
 Then, in any acyclic complex of objects from\/ $\sA$ with the objects
of cocycles belonging to\/ $\sC$, the objects of cocycles actually
belong to\/~$\sA'$. \par
\textup{(b)} Put\/ $\sB=\sS^{\perp_1}\cap\sC=\sA^{\perp_1}\cap\sC$
(so\/ $\sB'\subset\sB\subset\sC$).
 Let $A^\bu$ be a complex in\/ $\sK$ with the terms belonging to\/
$\sA$, and let $X^\bu$ be an acyclic complex in\/ $\sK$ with
the terms belonging to\/ $\sB$ and the modules of cocycles also
belonging to\/~$\sB$.
 Then any morphism of complexes $A^\bu\rarrow X^\bu$ is homotopic
to zero. 
\end{thm}

\begin{lem} \label{B-kappa-absolutely-pure-in-C}
 In the notation of Theorem~A or
Theorem~\ref{theorem-A-detailed-in-two-parts}(b), let\/
$0\rarrow B\rarrow L\rarrow C\rarrow0$ be a short exact sequence
in\/ $\sK$ with the terms $B$, $L$, $C\in\sC$.
 Assume that the object $B$ belongs to the class\/~$\sB$.
 Then the short exact sequence\/ $0\rarrow B\rarrow L\rarrow C
\rarrow0$ is $\kappa$\+pure in\/~$\sK$.
\end{lem}

\begin{proof}
 It is only important that $B\in\sS^{\perp_1}$, \ $L\in\sK$,
and $C\in\sC$.
 By Proposition~\ref{kappa-pure-exact-structure-subcategory}(b),
it suffices to check that any morphism $S\rarrow C$ with $S\in\sS$
lifts to a morphism $S\rarrow L$.
 This holds because $B\in\sB\subset\sS^{\perp_1}\subset\sK$.
\end{proof}

\begin{proof}[Proof of
Theorem~\ref{theorem-A-detailed-in-two-parts}(b)]
 The argument is similar to the proofs of
Theorem~\ref{theorem-0a-detailed-in-two-parts}(b)
and~\cite[Theorem~4.2]{BHP}.
 First of all, one has $\sS^{\perp_1}=\sA^{\perp_1}\subset\sK$ by
the Eklof lemma (Lemma~\ref{eklof-lemma}) applied in the abelian
category~$\sK$; so $\sS^{\perp_1}\cap\sC=\sA^{\perp_1}\cap\sC$.

 Without loss of generality we can assume that the terms of the complex
$A^\bu$ belong to $\Fil(\sS)$.
 Then, by Proposition~\ref{stovicek-general-hill-prop}, the complex
$A^\bu$ is filtered by (bounded below) complexes with the terms
belonging to~$\sS$.

 By Lemma~\ref{ext-homotopy-hom-lemma}, for any complex $A^\bu$ with
the terms in $\sA$ and any complex $B^\bu$ with the terms in $\sB$
we have an isomorphism of abelian groups
$$
 \Ext^1_{\bC(\sK)}(A^\bu,B^\bu[-1])\simeq
 \Hom_{\bH(\sK)}(A^\bu,B^\bu).
$$
 So, instead of showing that $\Hom_{\bH(\sK)}(A^\bu,X^\bu)=0$
as desired in the theorem, it suffices to prove that
$\Ext^1_{\bC(\sK)}(A^\bu,X^\bu[-1])=0$.
 Making use of the Eklof lemma (Lemma~\ref{eklof-lemma}) again,
the question reduces to showing that
$\Ext^1_{\bC(\sK)}(S^\bu,X^\bu[-1])=0$ for any complex $S^\bu$ with
the terms belonging to $\sS$ and any complex $X^\bu$ as in the theorem.
 Applying Lemma~\ref{ext-homotopy-hom-lemma} again, we conclude that it
suffices to show that any morphism of complexes $S^\bu\rarrow X^\bu$
is homotopic to zero.

 Finally, we observe that all $\kappa$\+presentable objects are
$\kappa$\+pure-projective in $\sK$ (by the definitions), while any
acyclic complex in $\sK$ with the objects of cocycles belonging
to $\sB$ is $\kappa$\+pure acyclic
(by Lemma~\ref{B-kappa-absolutely-pure-in-C}).
 Thus any morphism of complexes $S^\bu\rarrow X^\bu$ is homotopic to
zero by Theorem~\ref{kappa-neeman-stovicek-theorem}.
\end{proof}

\begin{proof}[Proof of
Theorem~\ref{theorem-A-detailed-in-two-parts}(a)]
 The assumption of deconstructibility presumes that the class $\sC$
is closed under transfinitely iterated extensions in~$\sK$.
 The class $\sC$ is also closed under direct summands, since it is
closed under $\kappa$\+direct limits by
Proposition~\ref{kappa-varinjlim-varinjlim-closed}.
 So we have $\sA\subset\sC$.

 We endow the full subcategory $\sC\subset\sK$ with the exact category
structure inherited from the abelian exact structure of~$\sK$.
 Then the class $\sS$ is generating in $\sC$ by
Corollary~\ref{kernel-of-coproduct-onto-direct-limit-kappa-version}.
 Moreover, the exact category $\sC$ is of Grothendieck type by
Theorem~\ref{deconstructible-class-is-grothendieck-type}, and
therefore it has enough injective objects by
Theorem~\ref{grothendieck-type-enough-injectives}.
 Therefore, $\sC\cap{}^{\perp_1}\sB'=\sC\cap{}^{\perp_{\ge1}}\sB'$
by Lemma~\ref{exact-category-hereditarily-generated}(a) applied
to the exact category $\sE=\sC$ and the class of objects $\sT=\sS$.

 We also have $\sA\subset{}^{\perp_1}\sB'$ by the Eklof lemma
(Lemma~\ref{eklof-lemma}).
 Hence $\sA\subset\sA'$.
 Now both the classes $\sA$ and $\sA'$ are generating in $\sC$, and
Lemma~\ref{exact-category-hereditarily-generated}(b) with
Proposition~\ref{efficient-exact-hereditarily-generated} tell that
$(\sA',\sB')$ is a hereditary complete cotorsion pair in~$\sC$.
 The hereditariness is important for our argument below.

 It is also worth noticing that $(\sA,\sB)$ is a (nonhereditary)
cotorsion pair in $\sC$ by Theorem~\ref{eklof-trlifaj-efficient-exact}
(since the class $\sS$ is generating in~$\sC$).
 The notation $\Fil(\sS)$ is unambiguous (means the same in $\sK$
and in~$\sC$) by Lemma~\ref{deconstructibility-transitive}.

 Let $A^\bu$ be an acyclic complex of objects from $\sA$ in $\sK$
with the objects of cocycles belonging to~$\sC$.
 Then $A^\bu$ is also an acyclic complex in the exact category~$\sC$.
 One can easily see that the objects of cocycles of $A^\bu$ belong
to $\sA'$ if and only if the complex of abelian groups
$\Hom_\sC(A^\bu,B)$ is acyclic for any object $B\in\sB'$.
 This holds because $(\sA',\sB')$ is a cotorsion pair in $\sC$,
or more specifically, because $\sA'=\sC\cap{}^{\perp_1}\sB'$.

 Now let $J^\bu$ be an injective resolution of the object $B$
in the exact category~$\sC$.
 So $0\rarrow B\rarrow J^0\rarrow J^1\rarrow J^2\rarrow\dotsb$ is
an acyclic complex in $\sK$ with the objects of cocycles belonging
to $\sC$ and the objects $J^n$ injective in~$\sC$.
 We observe that the objects of cocycles of the complex $J^\bu$
actually belong to $\sB'$, because all the injective objects of
$\sC$ belong to $\sB'$ and the class $\sB'$ is closed under
the cokernels of admissible monomorphisms in $\sC$ (as the cotorsion
pair $(\sA',\sB')$ in $\sC$ is hereditary).

 Denote by $X^\bu$ the acyclic complex $(B\to J^\bu)$.
 Then the complex of abelian groups $\Hom_\sC(A^\bu,X^\bu)$, i.~e.,
the direct product totalization of the bicomplex of Hom groups, is
acyclic by Theorem~\ref{theorem-A-detailed-in-two-parts}(b)
(which we have proved above).
 This holds because $A^\bu$ is a complex with the terms in $\sA$,
while $X^\bu$ is an acyclic complex with the terms in $\sB$ and
the objects of cocycles in~$\sB$ (recall that $\sB'\subset\sB$).

 On the other hand, the complex of abelian groups
$\Hom_\sC(A^\bu,J^\bu)$ is acyclic as well, since the complex
$A^\bu$ is acyclic in $\sC$ and $J^\bu$ is a bounded below complex
of injective objects in $\sC$ (cf.\ the proof of
Theorem~\ref{theorem-0a-detailed-in-two-parts}(a)).
 Since both the complexes $\Hom_\sC(A^\bu,X^\bu)$ and
$\Hom_\sC(A^\bu,J^\bu)$ are acyclic, and the complex $X^\bu$ has
the form $X^\bu=(B\to J^\bu)$, we can conclude that the complex
$\Hom_\sC(A^\bu,B)$ is acyclic.
\end{proof}

\begin{proof}[Proof of Theorem~A from Section~\ref{introd-theorem-A}]
 This is precisely the assertion of
Theorem~\ref{theorem-A-detailed-in-two-parts}(a).
\hbadness=2200
\end{proof}

\begin{cor} \label{theoremA-C-A-Aprime-periodicity-cor}
 Let\/ $\sK$ be a Grothendieck category, and let $\kappa$~be a regular
cardinal such that\/ $\sK$ is a locally $\kappa$\+presentable category.
 Let\/ $\sS\subset\sK$ be a class of (some) $\kappa$\+presentable
objects closed under transfinitely iterated extensions indexed by
ordinals smaller than~$\kappa$.
 Put\/ $\sC=\varinjlim^{(\kappa)}\sS\subset\sK$, and assume that
the class\/ $\sC$ is deconstructible in\/~$\sK$.
 Denote by\/ $\sA=\Fil(\sS)^\oplus$ the class of all direct summands of
transfinitely iterated extensions of objects from\/ $\sS$ in\/~$\sK$. 
 Put\/ $\sB'=\sS^{\perp_{\ge1}}\cap\sC$ and
$\sA'=\sC\cap{}^{\perp_1}\sB'=\sC\cap{}^{\perp_{\ge1}}\sB'$
(so\/ $\sA\subset\sA'\subset\sC$).
 Then, for any short exact sequence~\eqref{periodicity-sequence}
as in Section~\ref{introd-seven-items} with objects $L\in\sA$
and $M\in\sC$, one has $M\in\sA'$.
 In other words, any\/ $\sA$\+periodic object belonging to\/ $\sC$
actually belongs to\/~$\sA'$.
\end{cor}

\begin{proof}
 Follows from Theorem~\ref{theorem-A-detailed-in-two-parts}(a) by
Proposition~\ref{chop-or-splice-prop}\,(1)\,$\Rightarrow$\,(2) applied
to the abelian category $\sK$ and the classes of objects $\sL=\sA$,
\ $\sL'=\sA'$, \ $\sM=\sC$.
\end{proof}

 Finally, we use the opportunity to explicitly state the result
suggested in~\cite[Remark~4.11]{BHP} and deduce it from
the results of this paper.

\begin{cor}
 Let\/ $\sK$ be a locally $\kappa$\+presentable Grothendieck category.
 Denote by\/ $\sS$ the class of \emph{all} $\kappa$\+presentable
objects in~$\sK$.
 Put\/ $\sB=\sS^{\perp_1}$ and\/ $\sA={}^{\perp_1}\sB\subset\sK$; so
$\sA$ is the class of all direct summands of\/ $\sS$\+filtered objects
in\/~$\sK$.
 Furthermore, put\/ $\sB'=\sS^{\perp_{\ge1}}$ and\/
$\sA'={}^{\perp_1}\sB'={}^{\perp_{\ge1}}\sB'$;
so\/ $\sA\subset\sA'$ and $\sB\supset\sB'$.
 Then, for any short exact sequence~\eqref{periodicity-sequence}
as in Section~\ref{introd-seven-items} with objects $L\in\sA$
and $M\in\sK$, one has $M\in\sA'$; in other words, any\/
$\sA$\+periodic object in\/ $\sK$ belongs to\/~$\sA'$.
 In any acyclic complex of objects from\/ $\sA$ in\/ $\sK$,
the objects of cocycles belong to\/~$\sA'$.
\end{cor}

\begin{proof}
 One has $\sA=\Fil(\sS)^\oplus$ by
Theorem~\ref{eklof-trlifaj-theorem}(b), as the class $\sS$ is
generating in~$\sA$.
 Furthermore, by the definition of a locally $\kappa$\+presentable
category we have $\sK=\varinjlim^{(\kappa)}\sS$.
 So the class $\sC=\varinjlim^{(\kappa)}\sS=\sK$ is deconstructible
(in itself) by~\cite[Proposition~3.13]{Sto-ICRA}.
 Now the first assertion of the corollary is provided by
Corollary~\ref{theoremA-C-A-Aprime-periodicity-cor}, and
the second one by Theorem~A or
Theorem~\ref{theorem-A-detailed-in-two-parts}(a).
\end{proof}


\bigskip
\begin{thebibliography}{99}
\smallskip

\bibitem{AR}
 J.~Ad\'amek, J.~Rosick\'y.
   Locally presentable and accessible categories.
London Math.\ Society Lecture Note Series~189,
Cambridge University Press, 1994.

\bibitem{AR2}
 J.~Ad\'amek, J.~Rosick\'y.
   On pure quotients and pure subobjects.
\textit{Czechoslovak Math.\ Journ.}\ \textbf{54 (129)}, \#3,
p.~623--636, 2004.

\bibitem{AT}
 L.~Angeleri H\"ugel, J.~Trlifaj.
   Direct limits of modules of finite projective dimension.
\textit{Rings, modules, algebras, and abelian groups}, Proceedings
of the Algebra Conference in Venezia, Lecture Notes in Pure and Appl.\
Math.\ \textbf{236}, Marcel Dekker, New York--Basel, 2004, p.~27--44.

\bibitem{BCE}
 S.~Bazzoni, M.~Cort\'es-Izurdiaga, S.~Estrada.
   Periodic modules and acyclic complexes.
\textit{Algebras and Represent.\ Theory} \textbf{23}, \#5,
p.~1861--1883, 2020.  	\texttt{arXiv:1704.06672 [math.RA]}

\bibitem{BHP}
 S.~Bazzoni, M.~Hrbek, L.~Positselski.
   Fp\+projective periodicity.
\textit{Journ.\ of Pure and Appl.\ Algebra} \textbf{228}, \#3,
article ID~107497, 24~pp., 2024.  \texttt{arXiv:2212.02300 [math.CT]}

\bibitem{BPS}
 S.~Bazzoni, L.~Positselski, J.~\v St\!'ov\'\i\v cek.
   Projective covers of flat contramodules.
\textit{Internat.\ Math.\ Research Notices} \textbf{2022}, \#12,
p.~19527--19564, 2022.  \texttt{arXiv:1911.11720 [math.RA]}

\bibitem{BG}
 D.~J.~Benson, K.~R.~Goodearl.
   Periodic flat modules, and flat modules for finite groups.
\textit{Pacific Journ.\ of Math.}\ \textbf{196}, \#1, p.~45--67, 2000.

\bibitem{BBE}
 L.~Bican, R.~El Bashir, E.~Enochs.
   All modules have flat covers.
\textit{Bulletin of the London Math.\ Society} \textbf{33}, \#4,
p.~385--390, 2001.

\bibitem{BGP}
 D.~Bravo, J.~Gillespie, M.~A.~P\'erez.
   Locally type $\FP_n$ and $n$\+coherent categories.
\textit{Applied Categorical Struct.}\ \textbf{31}, \#2,
article~no.~16, 21~pp., 2023.  \texttt{arXiv:1908.10987 [math.CT]}

\bibitem{Bueh}
 T.~B\"uhler.
   Exact categories.
\textit{Expositiones Math.}\ \textbf{28}, \#1, p.~1--69, 2010.
\texttt{arXiv:0811.1480 [math.HO]}

\bibitem{CH}
 L.~W.~Christensen, H.~Holm.
   The direct limit closure of perfect complexes.
\textit{Journ.\ of Pure and Appl.\ Algebra} \textbf{219}, \#3,
p.~449--463, 2015.  \texttt{arXiv:1301.0731 [math.RA]}

\bibitem{CoFu}
 R.~Colpi, K.~R.~Fuller.
   Tilting objects in abelian categories and quasitilted rings.
\textit{Trans.\ of the Amer.\ Math.\ Soc.} \textbf{359}, \#2,
p.~741--765, 2007.

\bibitem{CCS}
 M.~Cort\'es-Izurdiaga, S.~Crivei, M.~Saor\'\i n.
   Reflective and coreflective subcategories.
\textit{Journ.\ of Pure and Appl.\ Algebra} \textbf{227}, \#5,
article~ID~107267, 43~pp., 2023.  \texttt{arXiv:2109.05111 [math.CT]}

\bibitem{CEG}
 M.~Cort\'es Izurdiaga, S.~Estrada, P.~A.~Guil Asensio.
   A model structure approach to the finitistic dimension conjectures.
\textit{Math.\ Nachrichten} \textbf{285}, \#7, p.~821--833, 2012.
\texttt{arXiv:0908.2751 [math.RA]}

\bibitem{CS}
 M.~Cort\'es-Izurdiaga, J.~\v Saroch.
  The cotorsion pair generated by the Gorenstein projective modules
and $\lambda$\+pure-injective modules.
Electronic preprint \texttt{arXiv:2104.08602 [math.RT]}

\bibitem{CoSt}
 P.~\v Coupek, J.~\v St\!'ov\'\i\v cek.
   Cotilting sheaves on Noetherian schemes.
\textit{Math.\ Zeitschrift} \textbf{296}, \#1--2, p.~275--312, 2020.
\texttt{arXiv:1707.01677 [math.AG]}

\bibitem{CB}
 W.~Crawley-Boevey.
   Locally finitely presented additive categories.
\textit{Communicat.\ in Algebra} \textbf{22}, \#5, p.~1641--1674, 1994.

\bibitem{Dr}
 V.~Drinfeld.
   Infinite-dimensional vector bundles in algebraic geometry:
an introduction.
\textit{The unity of mathematics}, Progress in Math.\ \textbf{244},
Birkh\"auser Boston, 2006, p.~263--304.
\texttt{arXiv:math.AG/0309155}

\bibitem{DG}
 N.~V.~Dung, J.~L.~Garc\'\i a.
   Additive categories of locally finite representation type.
\textit{Journ.\ of Algebra} \textbf{238}, \#1, p.~200--238, 2001.

\bibitem{ET}
 P.~C.~Eklof, J.~Trlifaj.
   How to make Ext vanish.
\textit{Bull.\ of the London Math.\ Soc.}\ \textbf{33}, \#1,
p.~41--51, 2001.

\bibitem{EFI}
 S.~Estrada, X.~Fu, A.~Iacob.
   Totally acyclic complexes.
\textit{Journ.\ of Algebra} \textbf{470}, p.~300--319, 2017.
\texttt{arXiv:1603.03850 [math.AC]}

\bibitem{GZ}
 P.~Gabriel, M.~Zisman.
   Calculus of fractions and homotopy theory.
Springer-Verlag, Berlin--Heidelberg--New York, 1967.

\bibitem{GR}
 J.~R.~Garc\'\i a Rozas.
   Covers and envelopes in the category of complexes of modules.
Chapman \& Hall/CRC Research Notes in Math., 407, Boca Raton, FL, 1999.

\bibitem{Gil0}
 J.~Gillespie.
   Kaplansky classes and derived categories.
\textit{Math.\ Zeitschrift} \textbf{257}, \#4, p.~811--843, 2007.

\bibitem{Gil}
 J.~Gillespie.
   Models for homotopy categories of injectives and Gorenstein
injectives.
\textit{Communicat.\ in Algebra} \textbf{45}, \#6, p.~2520--2545, 2017.
\texttt{arXiv:1502.05530 [math.CT]}

\bibitem{GT}
 R.~G\"obel, J.~Trlifaj.
   Approximations and endomorphism algebras of modules.
Second Revised and Extended Edition.
De Gruyter Expositions in Mathematics 41,
De Gruyter, Berlin--Boston, 2012.

\bibitem{GIT}
 P.~A.~Guil Asensio, M.~C.~Izurdiaga, B.~Torrecillas.
   Accessible subcategories of modules and pathological objects.
\textit{Forum Mathematicum} \textbf{22}, \#3, p.~485--507, 2010.

\bibitem{HT}
 D.~Herbera, J.~Trlifaj. 
   Almost free modules and Mittag-Leffler conditions.
\textit{Advances in Math.}\ \textbf{229}, \#6, p.~3436--3467, 2012.
\texttt{arXiv:0910.4277 [math.RA]}

\bibitem{Kra}
 H.~Krause.
   Functors on locally finitely presented additive categories.
\textit{Colloquium Math.} \textbf{75}, \#1, p.~105--132, 1998.

\bibitem{Len}
 H.~Lenzing.
   Homological transfer from finitely presented to infinite modules.
\textit{Abelian group theory (Honolulu, Hawaii)},
Lecture Notes in Math.\ \textbf{1006}, Springer, Berlin, 1983,
p.~734--761.

\bibitem{Neem}
 A.~Neeman.
   The homotopy category of flat modules, and Grothendieck duality.
\textit{Inventiones Math.}\ \textbf{174}, \#2, p.~255--308, 2008.

\bibitem{OR}
 U.~Oberst, H.~R\"ohl.
   Flat and coherent functors.
\textit{Journ.\ of Algebra} \textbf{14}, \#1, p.~91--105, 1970.

\bibitem{Pper}
 L.~Positselski.
   Abelian right perpendicular subcategories in module categories.
Electronic preprint \texttt{arXiv:1705.04960 [math.CT]}. 

\bibitem{Pal}
 L.~Positselski.
   Local, colocal, and antilocal properties of modules and complexes
over commutative rings.
\textit{Journ.\ of Algebra} \textbf{646}, p.~100--155, 2024.
\texttt{arXiv:2212.10163 [math.AC]}

\bibitem{Pacc}
 L.~Positselski.
   Notes on limits of accessible categories.
\textit{Cahiers de topol.\ et g\'eom.\ diff\'er.\ cat\'egoriques}
\textbf{LXV}, \#4, p.~390--437, 2024.
\texttt{arXiv:2310.16773 [math.CT]}

\bibitem{Plce}
 L.~Positselski.
   Locally coherent exact categories.
\textit{Appl.\ Categorical Struct.}\ \textbf{32}, \#4, article no.~20,
30~pp., 2024.  \texttt{arXiv:2311.02418 [math.CT]}

\bibitem{PPT}
 L.~Positselski, P.~P\v r\'\i hoda, J.~Trlifaj.
   Closure properties of $\varinjlim\mathcal C$.
\textit{Journ.\ of Algebra} \textbf{606}, p.~30--103, 2022.
\texttt{arXiv:2110.13105 [math.RA]}

\bibitem{PS1}
 L.~Positselski, J.~\v St\!'ov\'\i\v cek.
   The tilting-cotilting correspondence.
\textit{Internat.\ Math.\ Research Notices} \textbf{2021}, \#1,
p.~189--274, 2021.  \texttt{arXiv:1710.02230 [math.CT]} 

\bibitem{PS5}
 L.~Positselski, J.~\v St\!'ov\'\i\v cek.
   Coderived and contraderived categories of locally presentable
abelian DG\+categories.
\textit{Math.\ Zeitschrift} \textbf{308}, \#1, article no.~14, 70~pp.,
2024.  \texttt{arXiv:2210.08237 [math.CT]}

\bibitem{PS6}
 L.~Positselski, J.~\v St\!'ov\'\i\v cek.
   Flat quasi-coherent sheaves as directed colimits, and quasi-coherent
cotorsion periodicity.
\textit{Algebras and Represent.\ Theory} \textbf{27}, \#6,
p.~2267--2293, 2024.  \texttt{arXiv:2212.09639 [math.AG]}

\bibitem{Pre}
 M.~Prest.
   Purity, spectra and localisation.
Encyclopedia of Math.\ and its Appl., 121.
Cambridge University Press, 2009.

\bibitem{RG}
 M.~Raynaud, L.~Gruson.
   Crit\`eres de platitude et de projectivit\'e: Techniques
de ``platification'' d'un module.
\textit{Inventiones Math.} \textbf{13}, \#1--2, p.~1--89, 1971.

\bibitem{SaoSt}
 M.~Saor\'\i n, J.~\v St\!'ov\'\i\v cek.
   On exact categories and applications to triangulated adjoints
and model structures.
\textit{Advances in Math.}\ \textbf{228}, \#2, p.~968--1007, 2011.
\texttt{arXiv:1005.3248 [math.CT]}

\bibitem{Sar}
 J.~\v Saroch.
   Approximations and Mittag--Leffler conditions---the tools.
\textit{Israel Journ.\ of Math.}\ \textbf{226}, \#2, p.~737--756,
2018.  \texttt{arXiv:1612.01138 [math.RA]}

\bibitem{SarSt}
 J.~\v Saroch, J.~\v St\!'ov\'\i\v cek.
   Singular compactness and definability for $\Sigma$\+cotorsion and
Gorenstein modules.
\textit{Selecta Math.\ (New Ser.)}\ \textbf{26}, \#2, Paper No.~23,
40~pp., 2020.  \texttt{arXiv:1804.09080 [math.RT]}

\bibitem{ST0}
 J.~\v Saroch, J.~Trlifaj.
   Kaplansky classes, finite character and $\aleph_1$\+projectivity.
\textit{Forum Mathematicum} \textbf{24}, \#5, p.~1091--1109, 2012.

\bibitem{ST}
 J.~\v Saroch, J.~Trlifaj.
   Test sets for factorization properties of modules.
\textit{Rendiconti Semin.\ Matem.\ Univ.\ Padova} \textbf{144},
p.~217--238, 2020.  \texttt{arXiv:1912.03749 [math.RA]}

\bibitem{Sim}
 D.~Simson.
   Pure-periodic modules and a structure of pure-projective resolutions.
\textit{Pacific Journ.\ of Math.}\ \textbf{207}, \#1, p.~235--256, 2002.

\bibitem{Sten}
 B.~Stenstr\"{o}m.
   Rings of quotients. An Introduction to Methods of Ring Theory.
Die Grundlehren der Mathematischen Wissenschaften, Band~217.
Springer-Verlag, New York, 1975.

\bibitem{Sto0}
 J.~\v St\!'ov\'\i\v cek.
   Deconstructibility and the Hill Lemma in Grothendieck categories.
\textit{Forum Math.}\ \textbf{25}, \#1, p.~193--219, 2013.
\texttt{arXiv:1005.3251 [math.CT]}

\bibitem{Sto-ICRA}
 J.~\v St\!'ov\'\i\v cek.
   Exact model categories, approximation theory, and cohomology of
quasi-coherent sheaves.
\textit{Advances in representation theory of algebras}, p.~297--367,
EMS Ser.\ Congr.\ Rep., Eur.\ Math.\ Soc., Z\" urich, 2013.
\texttt{arXiv:1301.5206 [math.CT]}

\bibitem{Sto}
 J.~\v St\!'ov\'\i\v cek.
   On purity and applications to coderived and singularity categories.
Electronic preprint \texttt{arXiv:1412.1615 [math.CT]}.

\bibitem{SP}
 J.~\v St\!'ov\'\i\v cek, D.~Posp\'\i\v sil.
   On compactly generated torsion pairs and the classification of
co-$t$-structures for commutative noetherian rings.
\textit{Trans.\ of the Amer.\ Math.\ Soc.}\ \textbf{368}, \#9,
p.~6325--6361, 2016.  \texttt{arXiv:1212.3122 [math.CT]}

\end{thebibliography}
\end{document}